\documentclass[psamsfonts]{amsart}

\usepackage{amsmath, amssymb,amsfonts}
\usepackage[all,arc]{xy}
\usepackage{enumerate}
\usepackage[all]{xy}
\usepackage{mathrsfs}
\usepackage{array}
\usepackage{verbatim}

\newtheorem{thm}{Theorem}[subsection]

\newtheorem{cor}[thm]{Corollary}

\newtheorem{prop}[thm]{Proposition}
\newtheorem{lem}[thm]{Lemma}

\theoremstyle{definition}
\newtheorem{defn}[thm]{Definition}
\newtheorem{comp}[thm]{Computation}

\theoremstyle{remark}
\newtheorem{rem}[thm]{Remark}
\newtheorem{rems}[thm]{Remarks}

\makeatletter
\let\c@equation\c@thm
\makeatother
\numberwithin{equation}{section}

\begin{document}

\title{Fields of Rationality of Cusp Forms}

\author{John Binder}
\email{binderj@math.mit.edu}
\address{
	Department of Mathematics\\Massachusetts Institute of Technology\\77 Massachusetts Avenue\\Cambridge, MA, USA}
	

\newcommand{\NN}{\mathscr{N}}
\newcommand{\CC}{\mathbb{C}}
\newcommand{\DDD}{\mathscr{D}}
\newcommand{\HH}{\mathcal{H}}
\newcommand{\RR}{\mathbb{R}}
\newcommand{\RRR}{\mathscr{R}}
\newcommand{\FF}{\mathcal{F}}
\newcommand{\KK}{\mathscr{K}}
\newcommand{\UU}{\mathscr{U}}
\newcommand{\II}{\mathscr{I}}
\newcommand{\EE}{\mathscr{E}}
\newcommand{\GG}{\mathscr{G}}
\newcommand{\PZ}{\mathbb{P}_{\mathbb{Z}}}
\newcommand{\ZZ}{\mathbb{Z}}
\newcommand{\PP}{\mathscr{P}}
\newcommand{\PPP}{\mathscr{P}}
\newcommand{\SSS}{\mathcal{S}}
\newcommand{\LL}{\mathcal{L}}
\newcommand{\MM}{\mathscr{M}}
\newcommand{\AAA}{\mathbb{A}}
\newcommand{\GGG}{\mathscr{G}}
\newcommand{\AAAA}{\AAA}
\newcommand{\Df}{\mathcal{D}_F}
\newcommand{\QQ}{\mathbb{Q}}
\newcommand{\QQQ}{\mathscr{Q}}
\newcommand{\DD}{\mathscr{D}}
\newcommand{\OO}{\mathcal{O}}
\newcommand{\VV}{\mathscr{V}}
\newcommand{\pp}{\mathfrak{p}}
\newcommand{\qq}{\mathfrak{q}}
\newcommand{\faa}{\mathfrak{a}}
\newcommand{\mm}{\mathfrak{m}}
\newcommand{\IIII}{\mathcal{I}}
\newcommand{\JJ}{\mathscr{J}}
\newcommand{\weak}{\rightharpoonup}
\newcommand{\weaks}{\rightharpoonup^*}
\newcommand{\simga}{\sigma}
\newcommand{\linf}[2]{\left\langle {#1},\, {#2}\right\rangle}
\newcommand{\into}{\hookrightarrow}
\newcommand{\im}{\text{im}}
\newcommand{\lists}[3]{{#1}_1{#2}\ldots {#2}{#1}_{#3}}
\newcommand{\qr}[2]{\left(\frac{#1}{#2}\right)}
\newcolumntype{M}{>{$}c<{$}}
\newcommand{\alg}{\text{alg}}
\newcommand{\Spec}{\mathop{\mathrm{Spec}}\nolimits}
\newcommand{\Proj}{\mathop{\mathrm{Proj}}\nolimits}
\newcommand{\Tr}{\mathop{\mathrm{Tr}}\nolimits}
\newcommand{\Hom}{\mathop{\mathrm{Hom}}\nolimits}
\newcommand{\spa}{\mathop{\mathrm{sp}}\nolimits}
\newcommand{\rank}{\mathop{\mathrm{rank}}\nolimits}
\newcommand{\Pic}{\mathop{\mathrm{Pic}}\nolimits}
\newcommand{\image}{\mathop{\mathrm{Im}}\nolimits}
\newcommand{\tors}{\mathop{\mathrm{tors}}\nolimits}
\newcommand{\ord}{\mathop{\mathrm{ord}}\nolimits}
\newcommand{\Imag}{\mathop{\mathrm{Im}}\nolimits}
\newcommand{\trdeg}{\mathop{\mathrm{trdeg}}\nolimits}
\newcommand{\codim}{\mathop{\mathrm{codim}}\nolimits}
\newcommand{\hei}{\mathop{\mathrm{ht}}\nolimits}
\newcommand{\sgn}{\mathop{\mathrm{sgn}}\nolimits}
\newcommand{\Gal}{\mathop{\mathrm{Gal}}\nolimits}
\newcommand{\supp}{\mathop{\mathrm{supp}}\nolimits}
\newcommand{\Cl}{\mathop{\mathrm{Cl}}\nolimits}
\newcommand{\CaCl}{\mathop{\mathrm{Ca\,Cl}}\nolimits}
\newcommand{\Div}{\mathop{\mathrm{Div}}\nolimits}
\newcommand{\Sym}{\mathop{\mathrm{Sym}}\nolimits}
\newcommand{\coker}{\mathop{\mathrm{coker}}\nolimits}
\newcommand{\imag}{\mathop{\mathrm{im}}\nolimits}
\newcommand{\End}{\mathop{\mathrm{End}}\nolimits}
\newcommand{\Frob}{\mathop{\mathrm{Frob}}\nolimits}
\newcommand{\Ann}{\mathop{\mathrm{Ann}}\nolimits}
\newcommand{\Art}{\mathop{\mathrm{Art}}\nolimits}
\newcommand{\rec}{\mathop{\mathrm{rec}}\nolimits}
\newcommand{\Aut}{\mathop{\mathrm{Aut}}\nolimits}
\newcommand{\Ad}{\mathop{\mathrm{Ad}}\nolimits}
\newcommand{\nr}{\mathop{\mathrm{nr}}\nolimits}
\newcommand{\cond}{\mathop{\mathrm{cond}}\nolimits}

\newcommand{\Ind}{\mathop{\mathrm{Ind}}\nolimits}
\newcommand{\spann}{\mathop{\mathrm{span}}\nolimits}
\newcommand{\Vol}{\mathop{\mathrm{Vol}}\nolimits}
\newcommand{\Irr}{\mathop{\mathrm{Irr}}\nolimits}
\newcommand{\Res}{\mathop{\mathrm{Res}}\nolimits}
\newcommand{\genus}{\mathop{\mathrm{genus}}\nolimits}
\newcommand{\scusp}{\mathop{\mathrm{scusp}}\nolimits}
\newcommand{\PProj}{\mathbb{P}\mathop{\mathrm{roj}}\nolimits}
\newcommand{\cones}[3]{\langle v_{#1},\, v_{#2},\, v_{#3}\rangle}
\newcommand{\Rsheaf}{\mathscr{R}}
\newcommand{\Qsheaf}{\mathscr{Q}}
\newcommand{\Ksheaf}{\mathscr{K}}
\newcommand{\Hsheaf}{\mathscr{H}}
\newcommand{\Msheaf}{\mathscr{M}}
\newcommand{\Rei}{\mathcal{R}}
\newcommand{\Rie}{\Rei}
\newcommand{\hol}{\text{hol}}
\newcommand{\Nsheaf}{\mathscr{N}}
\newcommand{\unr}{\text{unr}}
\newcommand{\sHom}{\mathcal{H}om}
\newcommand{\smallmat}[4]{\left(\begin{smallmatrix}{#1} & {#2} \\ {#3} & {#4} \end{smallmatrix} \right)}
\newcommand{\twomat}[4]{\begin{pmatrix}{#1} & {#2} \\ {#3} & {#4} \end{pmatrix}}
\newcommand{\Proh}{\Proj}
\newcommand{\jj}{\mathfrak{j}}
\newcommand{\old}{\text{old}}
\newcommand{\bs}{\backslash}
\newcommand{\diam}[1]{\langle {#1} \rangle}
\newcommand{\Alg}{\textbf{Alg\,}}
\newcommand{\BB}{\mathcal{B}}
\newcommand{\Detla}{\Delta}
\newcommand{\iso}{\xrightarrow{\sim}}
\newcommand{\dep}{\mathop{\mathrm{dep}}\nolimits}
\newcommand{\ind}{\mathop{\mathrm{ind}}\nolimits}
\newcommand{\vol}{\mathop{\mathrm{vol}}\nolimits}
\newcommand{\tr}{\mathop{\mathrm{tr}}\nolimits}
\newcommand{\Stab}{\mathop{\mathrm{Stab}}\nolimits}
\newcommand{\St}{\mathop{\mathrm{St}}\nolimits}
\newcommand{\meas}{\mathop{\mathrm{meas}}\nolimits}
\newcommand{\disc}{\mathop{\mathrm{disc}}\nolimits}
\newcommand{\cusp}{\mathop{\mathrm{cusp}}\nolimits}
\newcommand{\spec}{\mathop{\mathrm{spec}}\nolimits}
\newcommand{\geom}{\mathop{\mathrm{geom}}\nolimits}
\newcommand{\pl}{\mathop{\mathrm{pl}}\nolimits}
\newcommand{\new}{\mathop{\mathrm{new}}\nolimits}

\newcommand{\Lie}{\mathop{\mathrm{Lie}}\nolimits}

\newcommand{\trace}{\mathop{\mathrm{trace}}\nolimits}
\newcommand{\mupl}{\widehat \mu^{\mathop{\mathrm{pl}}}}
\newcommand{\nupl}{\widehat \nu^{\mathop{\mathrm{pl}}}}
\newcommand{\mucusp}{\widehat \mu^{\cusp}}
\newcommand{\mudisc}{\widehat \mu^{\disc}}
\newcommand{\GL}{\mathop{\mathrm{GL}}\nolimits}
\newcommand{\PGL}{\mathop{\mathrm{PGL}}\nolimits}
\newcommand{\SL}{\mathop{\mathrm{SL}}\nolimits}
\newcommand{\EP}{\mathop{\mathrm{EP}}\nolimits}

\newcommand{\one}{\mathbf{1}}

\newcommand{\n}{\mathfrak{n}}
\newcommand{\nn}{\mathfrak{n}}
\newcommand{\ff}{\mathfrak{f}}
\newcommand{\oo}{\mathfrak{o}}
\newcommand{\ft}{\mathfrak{t}}
\newcommand{\dd}{\mathfrak{d}}
\newcommand{\fa}{\mathfrak{a}}
\newcommand{\XX}{\mathfrak{X}}

\newcommand{\wh}{\widehat}

\begin{abstract}
In this paper, we prove that for any totally real field $F$, weight $k$, and nebentypus character $\chi$, the proportion of Hilbert cusp forms over $F$ of weight $k$ and character $\chi$ with bounded field of rationality approaches zero as the level grows large. This answers, in the affirmative, a question of Serre. The proof has three main inputs: first, a lower bound on fields of rationality for admissible $\GL_2$ representations; second, an explicit computation of the (fixed-central-character) Plancherel measure for $\GL_2$; and third, a Plancherel equidsitribution theorem for cusp forms with fixed central character.  The equidistribution theorem is the key intermediate result and builds on earlier work of Shin and Shin-Templier and mirrors work of Finis-Lapid-Mueller by introducing an explicit bound for certain families of orbital integrals.\end{abstract}

\maketitle

\section{Introduction} 
Given a cuspidal Hecke eigenform $f$, define its field of rationality $\QQ(f)$ to be the number field generated by all its Fourier coefficients $a_n(f)$.

In \cite{Ser97}, Serre proved the following:
\begin{thm}[(Serre, 1997)] \label{thm1.1} Fix an even weight $k$, a prime $p$, and an integer $A \in \ZZ_{\geq 1}$. Let $(N_\lambda)$ be a sequence of levels coprime to $p$ with $N_\lambda \to \infty$. As $\lambda \to \infty$, the proporition of cusp forms of level $\Gamma_0(N_\lambda)$ whose field of rationality satisfies $[\QQ(f):\QQ] \leq A$ approaches $0$.
\end{thm}

The argument was as follows: first, he used a trace formula argument to show that, as $N_\lambda \to \infty$, the eigenvalues of $T_p$ are distributed according to the Plancherel measure on $[-2p^{\frac{k-1}{2}},\, 2p^{\frac{k-1}{2}}]$. He then noted that all points have measure zero, and that the set of Weil-$p$-integers of weight $k$ and degree at most $A$ is finite. In particular, the proportion of cusp forms with $[\QQ(a_p(f)):\QQ] \leq A$ must be asymptotically zero.

On page 89 that paper, Serre posited that his theorem could be extended to arbitrary sequences. It is our goal to answer Serre's question in the affirmative and extend the result in three directions. First, we look at Hilbert cusp forms over an arbitrary totally real field $F$. Second, instead of restricting to cusp forms with trivial character, we allow ourselves to look at forms of an arbitrary (fixed) character.  Third, we'll be able to look at cusp forms of either even or odd weight.

We'll fix here some notation that will be in use throughout the paper. Fix a totally real field $F$ with $[F:\QQ] = n$, a weight $k = (k_1,\ldots,\, k_n)$, and a level $\nn\subseteq \oo_F$; let $\chi: (\oo_F/\nn)^\times \to \CC^\times$ be a character. Let $B_{k}(\Gamma_1(\nn),\, \chi)$ be a basis of Hecke eigenforms of weight $k$, level $\Gamma_1(\nn)$, and character $\chi$. Fix moreover an integer $A \in \ZZ_{\geq 1}$. We define
$$B_k(\Gamma_1(\nn),\, \chi)_{\leq A} = \{f\in B_{k}(\Gamma_1(\nn),\, \chi) \mid [\QQ(f):\QQ] \leq A\}.$$

In our notation, Serre's theorem can be rephrased as follows: 
\begin{thm} \label{thm1.2}Let $F = \QQ$. Fix an auxiliary prime $p$, an even weight $k$, and an integer $A \geq 1$. Let $\nn_\lambda \to \infty$ be a sequence of levels with $(\nn_\lambda,\, p) = 1$ for all $\nn_\lambda$. Then
$$\lim_{\lambda \to \infty}\frac
	{B_k(\Gamma_1(\nn_\lambda),\,1)_{\leq A}}
	{B_k(\Gamma_1(\nn_\lambda),\,1)}
= 0$$
where $1$ denotes the trivial character.
\end{thm}

Let $\chi:(\oo_F/\nn)^{\times} \to \CC^\times$ be a character and $k$ be a weight. There is an obstruction to the existence of a cusp form of weight $k$ and character $\chi$. The weight $k$ determines the central character $\chi_{\infty}$ of the associated automorphic representation at the Archimedean places. As such, if such a cusp form exists, then there must be an automorphic character $Z(F) \bs Z(\AAA_F) \to \CC^\times$ that restricts to $\chi$ on $\wh \oo_F^\times$ and $\chi_{\infty}$ on $Z(F_\infty)$. If such an automorphic character exists, we say $\chi$ \emph{occurs in weight $k$}. For instance, when $F = \QQ$, a character $\chi$ occurs in weight $k$ if and only if $\chi(-1) = (-1)^k$. When $F \neq \QQ$ this requirement is more stringent because $\oo_F^\times$ is infinite. 

Our Main Theorem is:
\begin{thm}[(Theorem \ref{thm10.1})] \label{thm1.3} Fix a totally real field $F$, a weight $k = (k_1,\ldots,\, k_n)$, a character $\chi: F^\times \bs \AAA_F^\times \to \CC^{\times}$ of conductor $\ff$ occurring in weight $k$, and an integer $A \geq 1$. Let $(\nn_{\lambda})$ be any sequence of ideals with $\ff \mid \nn_\lambda$ and $N(\nn_\lambda) \to \infty$ as $\lambda \to \infty$. Then
$$\lim_{\lambda \to \infty} \frac
	{B_k(\Gamma_1(\nn_\lambda),\,\chi)_{\leq A}}
	{B_k(\Gamma_1(\nn_\lambda),\,\chi)}
= 0.$$
\end{thm}

The key intermediate result in our paper is the Plancherel equidistribution theorem:
\begin{thm}[(Plancherel equidistribution theorem, \ref{thm9.1})] \label{thm1.4}Fix $F$ and let $S$ be a finite set of finite places of $F$. Fix a discrete series representation $\pi_{\infty}$ of $\GL_2(F_{\infty})$ and let $\chi: F^\times \bs \AAA_F^\times \to \CC^\times$ be an automorphic character of conductor $\ff$ extending $\chi_{\pi_\infty}$.
Let $\nn_\lambda$ be a sequence of levels coprime to $S$, so that $\ff^S \mid \nn_\lambda$ and $N(\nn_\lambda) \to \infty$. 

As $\lambda \to \infty$, the $S$-components of cuspidal automorphic representations $\pi = \pi_S\otimes\pi^{S,\infty}\otimes \pi_{\infty}$, where $\pi^{S,\infty}$ has conductor dividing $\nn_\lambda$, and where $\chi_{\pi} = \chi$, are equidistributed according to the Plancherel measure $\mupl_{S,\chi}$, when counted with the appropriate multiplicity.
\end{thm}

Once this is proved, the Main Theorem follows by relating the field of rationality of certain local representations to their conductors (Proposition \ref{prop3.6}) and from explicit computations with the (fixed central character) Plancherel measure over $\GL_2(F_{\pp})$ (Computation \ref{comp6.13}).  Indeed, for a large prime and $\pp$ and a given conductor, $r\neq 1$ at $\pp$, we show that only a small proportion of automorphic representations (by Plancherel measure) of conductor $r$ have small field of rationality.  Thus, if we take a large enough prime and break our sequence of levels $(\nn_\lambda)$ into subsequences depending on $\ord_{\pp}(\nn_\lambda)$, we can use the Plancherel equidistribution theorem to handle each subsequence separately.  There is some difficulty with the $r = 1$ case, which we get around by taking a large set of large primes.  The details are the crux of the proof in Chapter 10.

We have stated the Plancherel equidistribution theorem in greater generality than necessary to prove the main theorem. Indeed, our main theorem is (conjecturally) vacuous in certain situations: for instance, if $k = (k_1,\ldots,\, k_n)$ and there is an $i,\,j$ with $k_i\not\equiv k_j$ mod $2$, then the associated representation is not $C$-algebraic and therefore, at least conjecturally, will not have a finite-degree field of rationality (see \cite[section 2]{ST13} for a discussion of $C$-algebraicity; the failure of $C$-algebraicity for mixed-parity cusp forms is basically \cite[Theorem 1.4 (2)]{RT11}). However, because the methods we use to prove the Plancherel equidistribution theorem are representation-theoretic in nature, we can prove it without any algebraicity assumptions.

We will briefly mention three papers that include results in this direction, and which are the inspiration for our ideas:
\begin{itemize} 
\item Shin proves an equidistribution theorem for Hilbert modular forms of level $\Gamma$, where $(\Gamma)$ is a sequence of open-compact subgroups that `converge to one' in the appropriate sense. For instance, if $(\nn_{\lambda})$ is a nested sequence of ideals of $\oo_F$ whose intersection is the zero ideal, then the sequence $(\Gamma(\nn_{\lambda}))$ converges to one, but the sequence $(\Gamma_0(\nn_{\lambda}))$ does not. However, his method is sufficiently general to extend to representations of other algebraic groups.
\item In \cite{ST12}, Shin and Templier prove an equidistribution theorem for representations of $G(\AAA_F)$ of increasing level when $G$ is a \emph{cuspidal} group. In \cite{ST13} they prove, as a corollary, that if $\nn_{\lambda}$ is a sequence of ideals with $\ord_\pp(\nn_\lambda) \to \infty$ for some prime $\pp$, then 
$$\lim_{\lambda \to \infty} \frac
	{B_k(\Gamma_1(\nn_\lambda), 1)_{\leq A}}
	{B_k(\Gamma_1(\nn_\lambda), 1)} = 0.$$
\item Finis, Lapid, and Mueller have done considerable work on the Limit Multiplicity Problem, which is analogous to our Plancherel equidstribution.  The primary difference is that we follow \cite{Shi12} and \cite{ST12} by fixing a discrete series representation at $\infty$ and examine the limit multiplicities only at finite places, whereas they look at limit multiplicities at infinite places.  In \cite{FLM14} they solved the Limit Multiplicity Problem for a large class of groups (specifically those satisfying properties (BD) and (TWN) as given in Section 5 of that paper.  These groups include $\GL_n$ and $\SL_n$).  Since this paper was released as a preprint they have solved the Limit Multiplicity Problem in even greater generality by reducing the restriction on the sequence of level subgroups (see \cite{FL15}); we discuss the relationship between our work and theirs more below.
\end{itemize}

The broad ideas for proving our Plancherel equidistribution theorem stem from the proofs of similar theorems in these papers. Like them, we will use the trace formula, Harish-Chandra's Plancherel theorem, and Sauvageot's density theorem. However, in our case it is necessary to adapt these existing tools suitably to our situation. We will need versions of the Harish-Chandra Plancherel theorem and the Sauvageot density theorem over local fields to the fixed-central-character setting, at least for $\GL_2$. We also need a fixed-central-character version of the trace formula.  For $\GL_2$, the fixed-central-character trace formula is classical and has been stated in \cite{Shi63}, \cite{GJ79}, \cite{KL06}, \cite{Pal12} and elsewhere; a more general invariant fixed-central-character version has been stated in \cite{Art02}.  Versions of Arthur's (non-fixed central character) trace formula in \cite{Art88} and \cite{Art89}, however, are significantly more `user friendly' in that it is easier to bound the noncentral terms.  Therefore, we will adapt the invariant trace formula to the fixed-central-character setting.

Once these are in place, a key step of our proof is a careful asymptotic estimation of the geometric side of the trace formula. Specifically, we will examine the asymptotic behavior of the geometric terms of the trace formula for characteristic functions of $\Gamma_0(\nn)$ as $N(\nn) \to \infty$. This builds on the work of Shin and Shin-Templier, who chose sequences of functions whose orbital integrals eventually vanished, and their constant-term computations were simplified because they used characteristic functions of \emph{normal} subgroups of the maximal compact subgroup $K^\infty$. The function $\one_{\Gamma_0(\nn)}$ has nonzero orbital integrals for many $\gamma \in \GL_2(F)$, but we will be able to bound these orbital integrals explicitly as $N(\nn) \to \infty$.

As part of their work on the Limit Multiplicity Problem, Finis and Lapid have obtained bounds on trace formula terms for characteristic functions of level subgroups (see, for instance \cite[Section 5]{FL13}).  This has allowed them to solve the Limit Multiplicity Problem for all groups satisfying their conditions (BD) and (TWN) and any sequence of level subgroups whose level approaches $\infty$.  Even though we give bounds only for $\Gamma_0$-level subgroups of $\GL_2$, we hope that our work will not appear redundant, for the following reasons: first, our method of achieving bounds through a careful analysis of Bruhat-Tits buildings is intuitively different from their methods, even if it is perhaps more difficult to generalize to higher-rank groups; second, we obtain a concrete description of the rate at which our terms approach zero.  It is also worth noting that because they do not work with cuspidal functions at $\infty$, they use a non-invariant version of the trace formula.

The outline of this paper is as follows. In section \ref{sec:2}, we discuss the tempered spectrum $\GL_2(L)^{\wedge, t}$ of $\GL_2(L)$ for a $p$-adic field $L$, and recall how it is naturally endowed with the structure of a disjoint union of countably many compact real orbifolds. In section \ref{sec:3}, we discuss fields of rationality of cusp forms and tempered representations. A necessary result (Proposition \ref{prop3.6}) is that if the residue characteristic of $L$ is sufficiently large and $\pi$ is a tempered representation of $\GL_2(L)$ with conductor at least $3$, then its field of rationality must be large.

In section \ref{sec:4}, we define the Plancherel transform $\widehat f$ for a function $f$ in various Hecke algebras.
In section \ref{sec:5}, we discuss Euler-Poincar\'e functions on $\GL_2(\RR)$; these will allow us to apply the trace formula to count cuspidal automorphic representations whose Archimedean component is a fixed discrete series representation. In section \ref{sec:6}, we state the necessary representation-theoretic fixed-central-character prerequisites for our proof of the Plancherel equidistribution theorem: the trace formula, the Plancherel formula, and Sauvageot's density theorem.  We do not prove these results until the appendix, since they follow from the standard (non-fixed central character) analogs in the literature from elementary abelian Fourier analysis.  Indeed, the proofs are not necessary on first reading.  We also give an explicit description of the fixed-central-character Plancherel measure for $\GL_2(L)$.

In section \ref{sec:7}, we introduce counting measures and construct explicit test functions whose Plancherel transforms count the cusp forms of fixed character, weight, and level. In section \ref{sec:8}, we show an asymptotic vanishing result for orbital integrals and constant terms. In section \ref{sec:9}, we use the results from section 6-8 to prove the Plancherel equidistribution theorem.

Finally, in section \ref{sec:10}, we prove our main theorem. The proof follows from the Plancherel equidistribution theorem and a careful assessment of the explicit (fixed-central-character) Plancherel measure on $\GL_2(F_\pp)$.

In the appendix (section \ref{sec:11}), we prove the fixed-central-character trace formula and necessary properties of the fixed-central-character Plancherel measure.

\subsection{Notation and Conventions} We will fix here the following conventions:
\begin{itemize} \item $F$ will always refer to a totally real field, and $L,\, L'$ will always refer to $p$-adic fields. $K$ will be reserved for compact subgroups of $\GL_2(R)$, where $R$ is a local field or an ad\`ele ring $\AAA_F$.
	\item The notation $x\mapsto\widehat x$ takes many uses, so we fix a convention here. We will reserve lower-case Greek letters $\widehat \phi,\, \widehat \psi$ for Plancherel transforms of elements of the Hecke algebra of $\GL_2$ (see Definition \ref{defn5.2}). Latin letters $\widehat f,\, \widehat h$ will always denote general functions in $\mathscr{F}_0(\GL_2^{\wedge})$ (see Definition \ref{defn7.4}). Both $\widehat f$ and $\widehat \phi$ are complex-valued functions on the unitary spectrum of $\GL_2$, but the former is more general.  Upper case Greek letters such as $\Phi$ and $\Psi$ are used to denote functions on a subgroup of the center of $\GL_2$. In this case, $\widehat \Phi$ and $\widehat \Psi$ will denote their \emph{Fourier} transforms as functions on a locally compact abelian group.
	\item Lower-case fraktur letters will refer to integral ideals in $F$ or $L$. $\oo_F,\, \oo_L$ will always refer to the ring of integers, and $\pp$ will always refer to a prime. $p$ will be reserved for rational primes.
	\item By a \emph{sequence of levels} $(\nn_\lambda)$, we mean a sequence $(\nn_\lambda)$ of ideals of $\oo_F$. We always assume $N(\nn_\lambda) \to \infty$.
	\item Given a representation $\pi$ of $\GL_2(L)$, the conductor $c(\pi)$ will take values $0, \, 1, \, 2\ldots$. For a representation of $\GL_2(\AAA_F)$, the conductor $\ff(\pi)$ will always be an \emph{ideal} in $\oo_F$. As such, if $\pi$ is a representation of $\GL_2(\AAA_F)$, then $\ff(\pi) = \prod_{\pp} \pp^{c(\pi_\pp)}$.
	\item All characters $\chi$, $\chi',\, \eta$, etc, will be \emph{unitary} characters, and if they are characters on the ad\`ele group $\AAA_F^\times$, they will be assumed to be trivial on $F^\times$. If $\pi$ is a representation of a $p$-adic or ad\`elic group, its central character will be denoted $\chi_\pi$. A character $\chi_0,\, \eta_0$, etc. will always refer to a character on the elements of absolute $1$. 
	\item $\xi$ will always be used to denote a finite-dimensional irreducible representation of $\GL_2(F_\infty)$. If $k = (k_1,\ldots,\, k_n)$ is a weight, then $\xi_k$ denotes the finite-dimensional complement of the discrete series representation associated to any cusp form of weight $k$; in particular, $\xi_k$ will decompose as a tensor product of irreducible representations of the form $\Sym^{k_i - 2}(\RR^2)\otimes |\det|^{\frac{-k_i-2}{2}}$.\end{itemize}

\subsection{Acknowledgements} I would like to thank Ruthi Hortsch, Padma Srinivasan, Nicolas Templier, and the anonymous referee for their helpful suggestions.

I am especially indebted to my adviser, Sug Woo Shin, for his unfailing support and friendship, without which this paper would have been impossible.

\section{The Tempered Spectrum of $\GL_2(L)$}\label{sec:2}

The goal of this section is to briefly recall some topological properties of the tempered spectrum of $\GL_2(L)$ where $L$ is a $p$-adic field. Throughout, $q$ will denote the cardinality of the residue field of $L$. We recall some definitions and preliminary results. Let $\GL_2(L)^{\wedge}$ denote the set of irreducible unitary admissible representations of $\GL_2(L)$ (up to isomorphism); in particular, if $\pi \in \GL_2(L)^{\wedge}$, its central character $\chi_\pi$ is unitary.

\begin{defn}\label{defn2.1} Let $G$ be a connected, reductive group over a $p$-adic field $L$ and let $\pi$ be an admissible representation of $G(L)$. We say $(\pi,\, V_\pi)$ is a \emph{discrete series representation} if the matrix coefficient $g \mapsto \linf{\pi(g)v}{\tilde v}$ is in $L^2(G(L)/Z(L))$ for every $v\in V_\pi,\, \tilde v \in V_{\widetilde\pi}$. 

We say $\pi$ is \emph{tempered} if, instead, every matrix coefficient lies in $L^{2 + \epsilon}(G(L)/Z(L))$ for every $\epsilon > 0$.\end{defn}

Throughout this paper, we will denote the set of unitary representations of $G(L)$ as $G(L)^{\wedge}$, and the set of tempered unitary representations of $G(L)$ as $G(L)^{\wedge,\, t}$.

For the rest of the subsection we assume $G = \GL_n$. The following results are classical:

\begin{prop}\label{prop2.2} A representation $\pi \in \GL_n(L)^{\wedge}$ is a discrete series representation if and only if it is a generalized Steinberg representation $\St(\sigma,\, m)$ for a unitary supercuspidal representation $\sigma \in \GL_d(L)^\wedge$, and $n = md$.

A representation $\pi\in \GL_n(L)^{\wedge}$ is tempered if and only if it is of the form
$$I_P^G(\pi_1 \otimes \ldots \otimes \pi_r)$$
where $\pi_i$ is a discrete series representation of $\GL_{n_i}(L)$, with $n = n_1 + \ldots + n_r$, and $I_P^G$ denotes normalized induction.
\end{prop}

Fix a standard parabolic $P$ with Levi subgroup $M$, and let $X_u(M)$ denote the group of unramified unitary characters of $M$. Then $X_u(M)$ acts on the set of discrete series representations $\omega$ of $M$ via $\chi \cdot \omega = \omega \otimes \chi$. Each orbit $\OO_M$ under the action of $X_u(M)$ naturally acquires the topology of a \emph{compact orbifold}, and as such the set of discrete series representations of $M$ acquires the topology of a countable union of disjoint compact orbifolds.

Denote by $\Theta$ the set of pairs $(M,\, \OO_M)$ where $\OO_M$ is an orbit of discrete series representations of $M$. Say two pairs $(M,\, \OO_M)$ and $(M',\, \OO_{M'}')$ \emph{associated} if there is an element $s \in W^G$, the Weyl group of $G$, such that $s\cdot M = M'$ and $s\cdot \OO_M = \OO'_{M'}$. 

The normalized induction functor gives a surjective map
$$\coprod_{\Theta/\text{assoc}} (M,\, \OO_M)\to \GL_n(L)^{\wedge, t}.$$
(The fact that $I_P^G(\omega)$ is irreducible when $\omega$ is a discrete series representation of $M$ follows from \cite[Theorem 4.2]{Zel80}; we note here that this does not hold for general reductive $p$-adic groups).

Moreover, for a given orbit $(M,\, \OO_M)$, the stabilizer 
$$\Stab(M,\,\OO_M) = \{s \in W^G/W^M: s\cdot M = M,\, s\cdot\OO_M = \OO_M\}$$
 acts on $\OO_M$. The map above descends to a \emph{bijection}
$$\coprod_{\Theta/\text{assoc}} (M,\, \OO_M)/\Stab(M,\, \OO_M)\xrightarrow\sim G^{\wedge, t}.$$
This gives the tempered spectrum of $G$ the structure of a countable disjoint union of compact orbifolds.

{Throughout this paper, we will use $\OO_M$ to refer to an orbit of discrete series representations of a Levi subgroup $M$ of $G$. We will use $\OO$ to refer to an orbit in $G^{\wedge, t}$; that is, $\OO$ will refer to the image of an orbit $(M,\, \OO_M)$ under the normalized induction functor.}

\subsection{Tempered Orbits of $\GL_2(L)$}
\label{subsec:2.1}
In this subsection, we'll recall some facts about the tempered orbits of $\GL_2(L)$. We'll follow the standard practice of writing $\sigma_1\times \sigma_2$ for $I_P^G(\sigma_1\otimes \sigma_2)$ when $\sigma_1 \otimes \sigma_2$ is a discrete series representation of a Levi subgroup $M$. $\pi$ is irreducible since $\sigma_1$ and $\sigma_2$ are unitary.

It is convenient to partition the set of orbits into four types:
\begin{enumerate} [Type (1):]
	\item $\OO$ consists of elements $\chi \times \chi'$, where $\chi\chi^{\prime -1}$ is unramified;
	\item $\OO$ consists of elements $\chi \times \chi'$, where $\chi\chi^{\prime -1}$ is ramified;
	\item $\OO$ consists of elements $\St(\chi)$ where $\chi$ is a character; and
	\item $\OO$ consists of supercuspidal representations $\pi$.
\end{enumerate}

It is worth recalling the following:

\begin{defn}\label{defn2.3} Let $L'/L$ be a quadratic extension, let $\psi_L$ be an additive character on $L$, and let $\eta$ be a multiplicative character on $L'$ that is not $\Gal(L'/L)$-invariant. The \emph{dihedral representation} $\pi_\eta$ of $\GL_2(L)$ is defined as follows. First, let $\omega^1_{\eta, \psi}$ be the Weil representation of $\SL_2(L)$ on the subspace of functions $f \in C_c^{\infty}(L')$ satisfying the transformation property
$$f(yv) = \eta(y)^{-1} f(v)\,\,\,\,\,\,\,\,\, \text{for all $v\in E,\, y\in \ker(N_{L'/L})$.}$$
Upgrade this to a representation $\omega_{\eta,\, \psi}$ of 
$$\GL_2(L)^{L'} = \{g\in \GL_2(L): \det(g) \in N_{L'/L}(L'^\times)\}$$
by setting
$$\left(\omega_{\eta,\psi}\twomat{a}{0}{0}{1}f\right)(v) = |a|_F^{1/2} \eta(b)f(bv),\,\,\,\,\,\,\,\,\,\, a = N_{L'/L}(b).$$
Let $\pi_{\eta} = \Ind_{\GL_2(L)^{L'}}^{\GL_2(L)}(\omega_{\eta,\psi})$; this is independent of the choice of additive character $\psi$.
\end{defn}

We have the following facts:

\begin{enumerate}[(i)]
	\item The central character of $\pi_{\eta}$ is $\chi_{L'/L}\cdot \eta|_{F^\times}$. Here $\chi_{L'/L}: L^\times \to \CC^\times$ is the unique nontrivial character whose kernel is $N_{L'/L}(L'^{\times})$.
	\item $\pi_{\eta} \cong \pi_{\eta'}$ iff $\eta$ and $\eta'$ are characters on the same quadratic extension $L'$, and $\eta$ and $\eta'$ are $\Gal(L'/L)$-conjugate.
	\item If the residue characteristic of $L$ is odd, then all supercuspidal representations of $\GL_2(L)$ are dihedral.
	\item If $\chi$ is a character of $L^\times$ then $\pi_{\eta} \otimes \chi \cong \pi_{\eta \otimes (\chi \circ N_{L'/L})}$. 
\end{enumerate}

Facts (i) and (iv) are on page 121 of \cite{Sch02}, while (iii) is on page 120. Fact (ii) follows by noting that $\pi_\eta$ corresponds to the irreducible Weil representation $I_{W(L')}^{W(L)}(\eta)$ under the Local Langlands correspondence.

When the residue characteristic is odd, we have the following characterization of the orbits:

\begin{prop}\label{prop2.4} Assume the residue characteristic of $L$ is $p > 2$.
\begin{enumerate} 
	\item The orbits of type (1) are in correspondence with characters $\chi_0: \oo_L^\times \to \CC^{\times}$.
	\item The orbits of type (2) are in correspondence with pairs of characters $\chi_0 \neq \chi_0': \oo_L^{\times} \to \CC^{\times}$.
	\item The orbits of type (3) are in correspondence with characters $\chi_0: \oo_L^{\times} \to \CC^\times$.
	\item The orbits of type (4) are in correspondence with pairs $({L'},\, \{\eta_0,\, \overline\eta_0\})$ where ${L'}/L$ is a quadratic extension, and $\{\eta_0,\, \overline\eta_0\}: \oo_{L'}^\times \to \CC^{\times}$ is a $\Gal({L'}/L)$-conjugate pair of characters with $\eta_0 \neq \overline{\eta}_0$.
\end{enumerate}
\end{prop}

\begin{proof} For the first statement, we note that $\chi_1 \times \chi_2$ and $\chi_1' \times \chi_2'$ are in the same orbit if $\chi_1'\chi_1^{-1}$ and $\chi_2'\chi_1^{-1}$ are unramified. Moreover, $\chi_1\chi_2^{-1}$ is unramified, so any two characters differ by an unramified twist. As such, $\chi_i$ and $\chi_i'$ all share the same restriction to $\oo_L^{\times}$: this determines a canonical bijection between orbits and characters $\chi_0:\oo_L^{\times} \to \CC^\times$.

The proofs of (2) and (3) are exactly the same.

For (4), because $p > 2$, every supercuspidal representation of $\pi$ of $\GL_2(L)$ is a {dihedral} representation, so there is a pair $({L'},\, \eta)$ as above such that $\pi = \pi_\eta$. The proof will follow once we show that, given characters $\eta,\, \eta'$ of ${L'}^\times$, then $\pi_\eta$ and $\pi_{\eta'}$ differ by an unramified twist if and only if $\eta$ and $\eta'$ differ by an unramified twist. On the one hand, assume $\eta = \theta\eta'$ for an unramified $\theta: {L'}^\times \to \CC^{\times}$. Since $\theta$ is unramified we can write $\theta = \chi\circ N_{{L'}/L}$ for an unramified character $\chi$; then we have
$$\pi_{\eta'} = \pi_{\eta \otimes(\chi \circ N_{{L'}/L})} \cong \pi_{\eta} \otimes \chi.$$

On the other hand, if $\chi$ is an unramified character of $L^\times$ and $\pi_{\eta'} = \pi_{\eta} \otimes \chi$, then $\pi_{\eta'} = \pi_{\eta \otimes (\chi\circ N_{{L'}/L})}$, and so $\eta (\chi \circ N_{{L'}/L}) = \eta'$ or ${\overline \eta'}$.

Therefore, the supercuspidal orbits are parameterized by pairs of $\{\eta,\, {\overline\eta}\}$ up to unramified twist, and giving a character up to unramified twist is the same as giving its restriction to $\oo_{L'}^\times$ as above, completing the proof.
\end{proof}

It follows immediately from above that if $\pi$ and $\pi'$ are in the same orbit, then $\chi_\pi|_{\oo_L^\times} = \chi_{\pi'}|_{\oo_L^{\times}}$. We define $\chi_\OO = \chi_\pi|_{\oo_L^{\times}}$ for any $\pi \in \OO$.

A list of conductors of tempered representations is given in \cite[p. 122]{Sch02}:
\begin{itemize} 
	\item If $\pi = \chi \times \chi'$, then $c(\pi) = c(\chi) + c(\chi')$.
	\item If $\pi = \St(\chi)$, then
		$$c(\pi) = \begin{cases}
			1 & \text{if } \chi \text{ is unramified}
			\\ 2 \cdot c(\chi) &\text{if } \chi \text{ is ramified.}
		\end{cases}$$
	\item If $\pi$ is the dihedral representation $\pi_{\eta}$, then
		$$c(\pi) = \begin{cases}
			2 \cdot c(\eta) & \text{ if $L'/L$ is unramified}
			\\ c(\eta) + 1 & \text{ if $L'/L$ is ramified.}
		\end{cases}$$
\end{itemize}
Since the conductor of a character $\chi$ or $\eta$ depends only on its restriction to $\oo_L^\times$ or $\oo_{L'}^\times$, we can make the definition:

\begin{defn}\label{defn2.5} Let $\OO$ be an orbit in $\GL_2(L)^{\wedge, t}$. We define its conductor $c(\OO)$ to be the conductor $c(\pi)$ for any $\pi \in \OO$.\end{defn}


\section{Preliminaries on Fields of Rationality}\label{sec:3}

Throughout, let $F$ be a totally real field.

\begin{defn} \label{defn3.1}Let $f$ be a Hilbert modular form over $F$ of level $\Gamma_1(\nn)$, weight $k$, and character $\chi$ that is a Hecke eigenform. Then $\QQ(f) \subseteq \overline \QQ$ is the field generated by all the Fourier coefficients of $f$.
\end{defn}

\begin{defn}\label{defn3.2} Fix a level $\nn$, a weight $k$, a character $\chi: (\oo_F/\nn)^\times \to \CC^\times$ occuring in weight $k$, and an integer $A \in \ZZ_{\geq 1}$. We denote by $B_k(\Gamma_{1}(\nn),\, \chi)$ a basis of normalized Hecke eigenforms in $S_k(\Gamma_1(\nn),\,\chi)$, and define
$$B_k(\Gamma_1(\nn),\, \chi)_{\leq A} = \{f \in B_k(\Gamma_1(\nn),\,\chi)\mid [\QQ(f):\QQ] \leq A\}.$$
\end{defn} 

\begin{defn} \label{defn3.3}Let $G$ be a reductive group over a $p$-adic field $L$ and let $\pi$ be an admissible $G(L)$-representation. The \emph{field of rationality} $\QQ(\pi)$ is the fixed field of the subgroup
$$\{\sigma\in \Aut(\CC): {^\sigma\pi}\cong \pi\}.$$

If $\pi$ is an automorphic representation of $G(\AAA_F)$, then $\pi$ decomposes as $\pi \cong \bigotimes_v \pi_v$, and $\QQ(\pi)$ is the compositum of the fields $\QQ(\pi_v)$ over the finite places $v$ of $F$.
\end{defn}

\begin{lem} \label{lem3.4}Let $f$ be a Hecke eigenform of weight $k$, level $\nn$, and character $\chi$, and let $\pi_f$ be the associated $\GL_2(\AAA_{F})$-representation. Then $\QQ(f) = \QQ(\pi_f)$.\end{lem}
\begin{proof} This is \cite[Theorem 1.4 (5)]{RT11}. We omit the proof.\end{proof}

\subsection{Fields of rationality of tempered orbits of $\GL_2(L)$} \label{subsec:3.1}
In this subsection, we switch back to the local theory. Let $p > 2$. We assume $L$ is a $p$-adic field whose residue field has cardinality $q$. Throughout, $\pi$ will denote an irreducible unitary admissible representation of $\GL_2(L)$.

\begin{defn} \label{defn3.5}Let $\OO$ be an orbit in $GL_2(L)^{\wedge, t}$. We define $\QQ(\OO)$ to be the intersection of all $\QQ(\pi)$ for $\pi \in \OO$.\end{defn}

The goal of this subsection is to prove the following, in analogy with Corollary 3.12 of \cite{ST13}:

\begin{prop} \label{prop3.6}Let the $p$ be the residue characteristic of $L$, and assume $p > 2A + 1$. Let $\OO$ be a tempered orbit of $\GL_2(L)$ of conductor at least $3$. Then $[\QQ(\OO): \QQ] > A$.\end{prop}

It is worth comparing this result to \cite[Corollary 3.12]{ST13}

We begin with three lemmas, which rely on the characterization of orbits given in Proposition \ref{prop2.4}.

\begin{lem}\label{lem3.7} Let $\OO$ be the supercuspidal orbit corresponding to $\eta_0:\oo_{L'}^\times\to \CC^\times$. Let $\eta_0(x) = \zeta$ for some $x\in\oo_{L'}^\times$. Then $[\QQ(\zeta):\QQ] \leq 2[\QQ(\OO):\QQ]$.\end{lem}
\begin{proof} Recall the construction of the dihedral representation in Definition 2.3, and let $\sigma\in \Aut(\CC)$; it is easy to check if $GL_2(L)^{L'}$ acts on $f$ via $\omega_{\eta, \psi}$, then it acts on $\sigma\circ f$ as $\omega_{\sigma\circ \eta, \,\sigma\circ\psi}$. This exhibits an isomorphism $^\sigma \omega_{\eta,\psi} \cong \omega_{\sigma\circ \eta, \,\sigma\circ\psi}$. 

But the representation $\pi_\eta$ is independent of the choice of $\psi$, so upon induction we get
$$^\sigma\pi_{\eta} = {^\sigma\pi_{\eta, \psi}} \cong \pi_{\sigma\circ \eta, \,\sigma\circ\psi} = \pi_{\sigma\circ\eta}.$$

As such, if $^\sigma \pi \cong \pi$ for some $\pi \in \OO$ then $\sigma$ permutes the character $\eta_0$ with its conjugate $\overline \eta_0$ under $\Gal({L'}/L)$. Therefore, $\sigma$ fixes $(\eta_0 + \overline \eta_0)(x)$ and $(\eta_0 \overline \eta_0)(x)$ for $x\in \oo_{L'}^\times$ so both these quantities are in $\QQ(\OO)$. As such, $\zeta = \eta_0(x)$ is a root of the polynomial
$$T^2 - (\eta_0 + \overline \eta_0)(x)T + (\eta_0 \overline \eta_0)(x) \in \QQ(\OO)[T]$$
and so $\zeta$ is of degree at most $2$ over $\QQ(\OO)$, completing the proof.
\end{proof}

\begin{lem} \label{lem3.8}Let $\OO$ be a Steinberg orbit corresponding to $\chi_0: \oo_L^\times \to \CC^\times$. Then $\QQ(\OO) \supseteq \QQ(\chi_0).$\end{lem}

\begin{proof} Let $t = \smallmat {t_1}{0}{0}{t_2}$, and let $f:G \to \CC$ satisfy $f(tu g) = \chi(t_1t_2) f(g)$ for some $\chi$ with $\chi|_{\oo_L^\times} = \chi_0$ Then $\sigma \circ f$ satisfies $\sigma \circ f(tug) = \sigma(\chi(t_1t_2))$, exhibiting an isomorphism between $^\sigma\St(\chi)$ and $\St(\sigma(\chi))$. Now the proof follows exactly as above.\end{proof}

\begin{lem}\label{lem3.9} Let $\OO$ be an orbit consisting of principal series representations corresponding to $\chi_0,\, \chi_0' :\oo_L^\times \to \CC^\times$, with $\chi_0(x) = \zeta$ for $x\in \oo_L^\times$. Then, $[\QQ(\zeta):\QQ] \leq 2[\QQ(\OO):\QQ]$.\end{lem}

\begin{proof} Assume $\pi = \chi \times \chi'$ and assume $^\sigma \pi \cong \pi$. Arguing as above, we have an isomorphism $^\sigma \chi \times \chi' \cong (\sigma(\chi)) \times (\sigma (\chi'))$ and therefore $^\sigma \chi \times \chi' \cong \chi \times \chi'$ if and only if $\sigma$ permutes $\chi$ and $\chi'$. As such, $\sigma$ fixes both $\chi + \chi'$ and $\chi \chi'$. Therefore, $\chi_0(x) + \chi_0'(x)$ and $\chi_0(x)\chi_0'(x)$ are in $\QQ(\OO)$ for all $x\in \oo_L^\times$.

If $\chi_0(x) = \zeta$ then $\zeta$ is a root of
$$T^2 - (\chi_0(x) + \chi'_0(x))T + \chi_0(x) \chi'_0(x) \in \QQ(\OO)[T].$$
In particular, $\zeta$ is of degree at most $2$ over $\QQ(\OO)$, completing the proof.\end{proof}

With these lemmas in hand, we can prove Proposition \ref{prop3.6}.

\begin{proof} From the discussion of conductors before Definition \ref{defn2.5}, we see that if $c(\OO) \geq 3$ then $\OO$ takes one of the following forms:
\begin{itemize} 
	\item $\OO$ is a supercuspidal orbit corresponding to $\eta_0:\oo_{L'}^\times \to \CC^\times$, with $c(\eta_0) \geq 2$
	\item $\OO$ is a Steinberg orbit corresponding to $\chi_0:\oo_L^\times \to \CC^\times$, with $c(\chi_0) \geq 2$
	\item $\OO$ is a principal-series orbit corresponding to $\chi_0 \times \chi_0'$, where $c(\chi_0) \geq 2$ (up to switching $\chi_0$ and $\chi_0'$). \end{itemize}
In the second two cases, $\chi_0$ is nontrivial on $1 + \varpi_L\oo_L$, a pro-$p$-group and so $\zeta_p \in c(\chi_0)$. In the first case, $\eta_0$ is nontrivial on $1 + \varpi_{L'} \oo_{L'}$, again a pro-$p$ group, so $\zeta_p \in \QQ(\eta_0)$. Therefore, in all cases, $[\QQ(\OO):\QQ] \geq \frac{1}{2} [\QQ(\zeta_p): \QQ] = \frac{p-1}{2} > A$.
\end{proof}

\section{Fixed-Central Character Hecke Algebras and Plancherel Transforms}\label{sec:4}

In this section, we briefly introduce fixed-central-character Hecke algebras and Plancherel transforms.

Throughout this section, $F$ is a totally real field with $[F:\QQ] = n$ and $\AAA$ is the ring of ad\`eles over $F$. $R$ will be used to denote $\AAA$ or $F_v$ for some place $v$ of $F$.

The definitions and lemmas below will depend upon a choice of Haar measure. Fix once and for all the following Haar measures:
\begin{itemize} 
	\item If $L$ is a $p$-adic field, and $G(L)$ the group of $L$-points of some reductive group, the we choose the Haar measure giving a maximal compact subgroup measure $1$.
	\item We choose the Euler-Poincar\'e measure on $\GL_2(\RR)$ (see section 5) and the standard Haar measure on $\RR^\times$.
	\item On an ad\`elic group such as $\GL_2(\AAA_F)$ or $\AAA_F^\times$, we take the product measure of the local measures just described.
\end{itemize}

\begin{defn} \label{defn4.1}Let $\XX$ be a closed subgroup of the center $Z(\GL_2(R))$, and let $\chi: \XX \to \CC^{\times}$ be a unitary character. The \emph{Hecke algebra} $\HH(\GL_2(R),\, \XX,\,\chi)$ is the convolution algebra of smooth functions $\phi: \GL_2(R) \to \CC$ that are compactly-supported modulo $\XX$ and that satisfy the transformation property 
$$\phi(gx) = \phi(g)\chi(x)^{-1}\text{ for all $g\in G,\, x\in \XX$.}$$
\end{defn}

\begin{defn} \label{defn4.2}Let $\phi \in \HH(\GL_2(R),\, \XX,\,\chi)$. We define its \emph{Plancherel transform} $\widehat \phi$ as a complex function on the space of representations $\pi$ with $\chi_\pi|_{\XX} = \chi$, by
$$\widehat \phi(\pi) = \tr_{\XX} \pi(\phi) = \tr\left(v \mapsto \int_{\XX \bs\GL_2(R)} \phi(g)\pi(g)v\, dg\right).$$

The integrand is well-defined since $\phi(gx)\pi(gx) = \phi(g)\chi^{-1}(x) \chi(x) \pi(g) = \phi(g)\pi(g)$ for all $g\in G,\, x\in \XX$.
\end{defn}

We repeat here a piece of notation that will be in effect throughout the paper.  Greek letters like $\phi$ and $\psi$ are reserved for functions in some Hecke algebra, and then $\wh \phi,\, \wh \psi$ will be denote their Plancherel transforms on the tempered spectrum.  Latin letters like $\wh f,\, \wh h$ will be used to denote elements of the larger set $\FF_0$ of complex-valued functions on $\GL_2(L)^{\wedge}$; see Definition \ref{defn6.7}.

\section{Euler-Poincar\'e measures and Euler-Poincar\'e functions}
\label{sec:5}

Let $\xi$ be an irreducible, finite-dimensional representation of $\GL_2(F_\infty)$ and let $\pi_{\xi}$ be its discrete-series complement: that is, for every $v\mid \infty$, $\xi_v + \pi_{\xi, v}$ is equivalent to an induced representation in the Grothendieck group. In this section, we will prove the existence of a function $\phi_{\xi} \in \HH(\GL_2(F_\infty),\, Z(F_\infty),\, \chi_{\xi})$ such that for any \emph{infinite-dimensional} representation $\pi'$ of $\GL_2(F_\infty)$,

\begin{equation} 
\label{eq5.1}
\tr_{Z(F_\infty)}\pi'(\phi_\xi,\, \mu^{\EP}) = \begin{cases} 
	(-1)^{[F:\QQ]} & \pi' \cong \pi_{\xi}
	\\ 0 & \text{otherwise;}
\end{cases}
\end{equation}
here the trace is taken with respect the \emph{Euler-Poincar\'e measure} on $\GL_2(F_\infty)/Z(F_\infty)$:

\begin{defn} \label{defn5.2}Let $\overline{G}$ be the compact inner form of $\GL_2(\RR)/Z(\RR)$, and let $\overline \mu^{\EP}$ be the Haar measure on $\overline{G}$ of total measure $1$. We define the \emph{Euler-Poincare} measure on $\GL_2(\RR)/Z(\RR)$ as the unique Haar measure such that the induced measure on $\overline{G}$ is $\overline{\mu}^{\EP}$.

The Euler-Poincare measure on $\GL_2(F_\infty)/Z(F_\infty)$ is given by the product measure under the identification
$$\frac{\GL_2(F_\infty)}{Z(F_\infty)} 
	\cong \prod_{v\mid \infty} \frac{\GL_2(F_v)}{Z(F_v)} 
	\cong \prod_{v\mid \infty} \frac{\GL_2(\RR)}{Z(\RR)}.$$
\end{defn}

To construct $\phi_\xi$, it's enough to have local functions $\phi_{\xi_v}$ and let $\phi_\xi = \prod_v \phi_{\xi_v}$.

Let $K_{v}' = F_{v, > 0} \cdot O(2)_v \subseteq \GL_2(F_v)$.
For an irreducible finite-dimensional representation $\xi_v$ of $\GL_2(F_v)$ and an admissible representation $\pi_v$ such that $\xi_v$ and $\pi_v$ have the same central character on $F_{v, >0}$, we define the Euler-Poincar\'e characteristic:
$$\chi_{\EP}(\pi_v \otimes \xi_v^{\vee}) = \sum_{i \geq 0} (-1)^i \dim H^i(\Lie \GL_2(F_v),\, K_{v}',\, \pi_v \otimes \xi_v^{\vee});$$
(here the cohomology is $(\mathfrak{g}, K)$-Lie algebra cohomology).

Clozel and Delorme \cite[Th\'eor\`eme 3]{CD90} have constructed a function $\phi_{\xi_v} \in \HH(\GL_2(F_v),\, F_{v, >0},\, \chi_{\xi_v})$ such that
$$\tr_{F_{v, > 0}} \pi_v(\phi_{\xi_v},\, \mu^{\EP}) = \chi_{EP}(\pi_v \otimes \xi_v^{\vee}).$$
It is moreover well-known that $\chi_{\EP}(\pi_v \otimes \xi_v^{\vee}) = 0$ unless $\pi_v$ has the same infinitesimal character as $\xi_v$ (see the bottom of page 43 of \cite{ST12}). Since $\pi_v$ and $\xi_v$ also have the same central character (restricted to $F_{v, >0}$), then the Langlands classification for admissible representations of $\GL_2(\RR)$ tells us that if $\tr \pi_v(\phi_{\xi_v}) \neq 0$, then $\pi_v$ must be of one of the following three forms:
\begin{itemize}
	\item $\pi_v = \xi_v$
	\item $\pi_v$ is the discrete series complement of $\xi_v$; i.e., there is an exact sequence
		$$0 \to \xi_v \to \mu_1 \times \mu_2 \to \pi_v \to 0$$
		where $\mu_1 \times \mu_2$ is the representation induced from the character $\mu_1\otimes \mu_2$ on the Borel subgroup.
	\item If $\mu_1,\, \mu_2$ is as above, then $\pi_v = \mu_1 \times (\mu_2 \cdot \sgn)$ or $\pi_v = (\mu_1 \cdot \sgn) \times \mu_2$.
\end{itemize}
However, in the third case, $\pi_v$ is in the continuous series, and since $\tr \pi'_v (\phi_{\xi_v}) = 0$ for all other continuous-series representations $\pi'_v$, then we must have $\tr \pi_v(\phi_{\xi_v}) = 0$. We have therefore proved:
\begin{prop} \label{prop5.3}Assume $\pi_v'$ is infinite-dimensional, that $\chi_{\EP}(\pi_v'\otimes \xi_v^{\vee}) \neq 0$, and that $\chi_{\pi_v'}$ and $\chi_{\xi_v}$ agree on $A_{G,\infty}$. Then $\pi_v'$ is the discrete-series complement of $\xi_v$. 
\end{prop}

If $\pi_{\xi_v}$ is the discrete-series complement of $\xi_v$, then $\tr_{F_{v, >0}} \pi_{\xi_v}(\phi_{\xi_v}) = -1$ (see the fact at the top of page 44 of \cite{ST12}). By replacing $\phi_{\xi_v}$ with $g\mapsto \frac{1}{2} (\phi_{\xi_v}(g) + \chi_{\xi_v}(-1)\phi_{\xi_v}(-g))$, we may assume $\phi_{\xi_v} \in \HH(\GL_2(F_v),\, Z(F_v),\, \chi_{\xi_v})$. In this case we have
$$\tr_{F_v^\times} \pi_{\xi_v}(\phi_{\xi_v}) = -1.$$
Here we are making a choice of Haar measure that will be in effect for the rest of the paper: the Haar measure on $\GL_2(F_v)/Z(F_v)$ is chosen so that the finite group $Z(F_v)/F_{v, >0}\cong \{\pm 1\}$ gets total measure $1$, the measure on $\GL_2(F_v)/F_{v, >0}$ is the Euler-Poincar\'e measure, and the measures are compatible under
$$1 \to 
	\frac{Z(F_v)}{F_{v, >0}}
	\to \frac{\GL_2(F_v)}{F_{v, >0}}
	\to \frac{\GL_2(F_v)}{Z(F_v)}
	\to 1.$$

We will need later that $\phi_{\xi_v}(1) = -\dim(\xi_v)$. This basically follows from the Plancherel theorem for \emph{real} groups, and is proven at the bottom of p. 276 in \cite{Art89}.

Let $\xi = \bigotimes_v \xi_v$, and let $\phi_{\xi} = \prod_v \phi_{\xi_v}$. Its discrete-series complement is $\bigotimes_v \pi_{\xi_v}$. We have proven the following:
\begin{cor} \label{cor5.4}Let $F$ be a totally real field and let $\xi$ be an irreducible finite-dimensional representation of $\GL_2(F_\infty)$, whose complementary discrete series representation is $\pi_{\xi}$. Then there is a function $\phi_{\xi} \in \HH(\GL_2(F_\infty),\, Z(F_\infty),\, \chi_{\xi})$ such that
	\begin{itemize} 
		\item for any generic representation $\pi$ of 
		$\GL_2(F_\infty)$, 
		$$\tr_{Z(F_\infty)} \pi(\phi_{\xi}) = 
		\begin{cases} 
			(-1)^{[F:\QQ]} & \text{if }\pi = \pi_{\xi}
			\\ 0 & \text{otherwise.}\end{cases}$$
			\item $\phi_{\xi}(1) = (-1)^{[F:\QQ]}\dim \xi$.
	\end{itemize}
\end{cor}
\begin{proof} The only point that needs to be made is that $\xi$ is generic (i.e. has a Whittaker model) if and only if it is infinite-dimensional at every place.\end{proof}

\section{Representation-Theoretic Results for Fixed Central Character}
\label{sec:6}

In this section, we will state three important representation-theoretic results for fixed central character: first, a simple version of the invariant Trace formula; second, a description of the fixed-central-character Plancerhel measure; and third, a fixed-central-character version of Sauvageot's density theorem.  To our knowledge, these results as stated are not explicitly written down  in the literature, though they are known to the experts.  

We make a brief note on the proofs of these results.  The results can be derived from the non-fixed central character versions stated in the literature with abelian Fourier analysis; this is the tack we will take.  Because the proofs are long but elementary, we have decided to put them in the appendix; we will simply state the results here.

For the fixed-central-character trace formula, it is worth noting that the versions of the trace formula for $\GL_2$ stated, for instance, in \cite[(7.14)-(7.19)]{GJ79}, \cite{Shi63},\cite[Theorem 22.1]{KL06} and \cite{Pal12} are all fixed-central-character versions.  However, we believe it is easiest to take the version from \cite{Art02} as a starting point since it fits most nicely into the framework of \cite{Art88} and \cite{Art89}, and the geometric terms of trace formulae in these papers are easiest to manage.


\subsection{Fixed-Central-Character Invariant Trace Formula for $\GL_2$}
\label{subsec:6.1}

We begin with a definition:

\begin{defn} \label{defn6.1} Let $\phi: \GL_2(\AAA) \to \CC$ be smooth and compactly-supported modulo the center.  
\begin{itemize}
	\item Let $\gamma \in \GL_2(\AAA)$ and let $G_{\gamma}(\AAA)$ be its centralizer in $\GL_2(\AAA)$.  We define the \emph{orbital integral} 
		$$O_{\gamma}(\phi) = \int_{G_\gamma(\AAA)\bs G(\AAA)} \phi(g^{-1}\gamma g)\,dg$$
	\item Let $\gamma \in T(\AAA^\infty)$, the torus of diagonal elements.  We define the \emph{constant term}
		$$Q_{\gamma}(\phi) = \int_{K^\infty} \int_{\AAA^\infty} \phi\left(k^{-1}\gamma \twomat{1}{a}{0}{1} k\right)\,da\,dk.$$
\end{itemize}
\end{defn}

If $\phi$ is a product of local functions, the the constant terms and orbital integrals decompose as products of local constant terms and local orbital integrals.

\begin{defn} 
\label{defn6.2}
	Let $\gamma_v\in \GL_2(F_v)$.  We say $\gamma_v$ is \emph{elliptic} if it is semisimple and the split component of the center of the centralize $G_{\gamma_v}$ is $A_G(F_v)$. Equivalently, in the case of $\GL_2$, $\gamma_v$ is either central, or it is semisimple but not diagonal in $\GL_2(F_v)$.
	Let $\phi = \prod_v$ be smooth and compactly-supported modulo the center.  We say $\phi$ is \emph{cuspidal} at a place $v$ if for every element $\gamma_v \in \GL_2(F_v)$ that is not elliptic, the orbital integral $O_{\gamma_v}(\phi_v)$ vanishes.
\end{defn}

Here we note that the Euler-Poincare functions $\phi_\xi$ at $\infty$ from the previous section are cuspidal (see, for instance, page 267 of \cite{Art89}). this will allow us to use simpler forms of the trace formula.

\begin{prop}[(Fixed-Central-Character Invariant Trace Formula)] \label{prop6.3} Let $F$ be a totally real field and let $\AAA$ be its ad\`ele ring.  Let $\chi$ be an automorphic character of $\AAA^\times$.  Let $\phi = \phi^\infty \phi_{\xi} \in \HH(\GL_2(\AAA), Z(\AAA), \chi)$, where $\phi_\xi$ is an Euler-Poincare function as in Corollary \ref{cor5.4}.
\begin{itemize}
	\item If $F = \QQ$ then 
		\begin{align*} \sum_{\pi} \tr_{Z}(\phi) & = 
			\vol(G(F)Z(\AAA)\bs G(\AAA)) \phi(1)
			\\ & + \sum_{\substack{
				\gamma \in (G(F) - Z(F))/Z(F) 
				\\ \gamma \text{ semisimple}
				\\ \gamma_{\infty} \text{ elliptic}}} 
			C(G,\, \gamma) \vol(Z(F)A_{G,\infty}\bs Z(\AAA)) O_{\gamma}(\phi)
			\\ & +  \sum_{\gamma \in T(F)/Z(F)} C(T,\,\gamma) \vol(Z(F)A_{G,\infty}\bs Z(\AAA))Q_{\gamma}(\phi)
		\end{align*}
	\item If $F \neq \QQ$ then
		\begin{align*} \sum_{\pi} \tr_{Z}(\phi) & = 
			\vol(G(F)Z(\AAA)\bs G(\AAA)) \phi(1)
			\\ & + \sum_{\substack{
				\gamma \in (G(F) - Z(F))/Z(F) 
				\\ \gamma \text{ semisimple}
				\\ \gamma_{\infty} \text{ elliptic}}} 
			C(G,\, \gamma) \vol(Z(F)A_{G,\infty}\bs Z(\AAA)) O_{\gamma}(\phi)
		\end{align*}
\end{itemize}
Here $C(G,\gamma),\, C(T,\, \gamma)$ are constants that depend only on $\gamma$ and not on $\phi$.
\end{prop}

It will be useful to name the expressions in the above equation.  We denote the left-hand, or \emph{spectral} side, as $I_{\spec}(\phi,\,Z(\AAA),\,\chi)$.  The right hand, or \emph{geometric} side, we will denote by $I_{\geom}(\phi,\, Z(\AAA),\,\chi)$.

\begin{rems} \label{rems6.4} The exact values of the constants $C(G,\,\gamma)$ and $C(T,\,\gamma)$ are unnecessary for our purposes since we will show that these terms vanish asymptotically.  The interested reader can see Theorem 6.1, and the subsequent remark, in \cite{Art89}, or (4.2), (4.3), and (4.4) of \cite{Shi12}.
\end{rems}

\begin{proof}
Only a sketch will be necessary, since the leg work has been done in \cite{Art88}, \cite{Art89}, and \cite{Art02} (in fact, we will simply piece these results together).  To this end, consider the fixed central character invariant trace formula given in \cite{Art02}.  The geometric side is given in \cite{Art02}[Proposition 2.2]; this contains the same terms as the geometric side of the trace formula given in \cite{Art88}, except that the sum is over conjugacy classes \emph{modulo center}.  The spectral side given in \cite{Art02}[Proposition 3.1] and matches the spectral side in \emph{Art88}, except that it restricts to the set of representations where the central character is fixed.

With this in hand, the versions of the trace formula given above follow exactly as the proofs of the non-fixed central character analogs.  For the first version, we can follow the arguments of sections 2-6 of \cite{Art89} to discern (i) as the analog of Theorem 6.1 there.  The second version follows similarly as an analog of \cite{Art88}[Corollary 7.5].

\end{proof}

\subsection{The Fixed-Central-Character Plancherel Measure}
\label{subsec:6.2}

We now turn away from the trace formula to a pair of local results involving the Plancherel measure.  For reference, we begin with the following result, following \cite{Wal03}:

\begin{thm}[(Harish-Chandra's Plancerhel Theorem)] \label{thm6.5} Let $L$ be a local field and let $\GL_2(L)^{\wedge,t}$ be the tempered spectrum of $\GL_2(L)$.  Given a measure on $\GL_2(L)$, there is a unique measure $\mupl$ on $\GL_2(L)^{\wedge,t}$, called the \emph{Plancherel measure}, such that, for any function $\phi\in C_c^\infty(\GL_2(L))$, we have
$$\phi(1) = \int_{\GL_2(L)^{\wedge, t}} \wh\phi(\pi)\,d\mupl(\pi).$$
\end{thm}

We also have the following density theorem of Sauvageot \cite[Thm 7.3]{Sau97}:
\begin{thm}\label{thm6.6}[(Sauvageot's Density Theorem)] \label{thm6.6} Let $\wh f: \GL_2(L)^{\wedge} \to \CC$ be supported on a finite number of Bernstein components and assume it is continuous outside a set of Plancherel measure zero.  Given $\epsilon > 0$, there are functions $\phi,\, \psi \in C_c^\infty(\GL_2(L))$ such that
\begin{enumerate}[(i)]
	\item $|\wh f(\pi) - \wh\phi(\pi)| \leq \wh\psi(\pi)$ for all $\pi \in \GL_2(L)^{\wedge}$, and
	\item $\mupl(\wh\psi) < \epsilon$
\end{enumerate}
\end{thm}

In view of this theorem, we make the following definition:
\begin{defn} \label{defn6.7} The set $\mathscr{F}_0(\GL_2(L)^\wedge)$ is the set of complex-valued function that are supported on a finite number of Bernstein components and that are continuous outside a set of Plancherel measure zero.  
\end{defn}

Indeed, Sauvageot proves that the function in $\FF_0$ are \emph{precisely} those for which Sauvageot's density theorem holds, but we will not need this fact here.

\begin{rem} \label{rem6.8} We note that Harish-Chandra's Plancherel Theorem and Sauvageot's density theorem apply equally well to a finite set of places, in the following sense: if $F$ is a global field and $S$ a finite set of finite places, we may replace the $L$ in the statements above with $F_S$.
\end{rem}

In order to state the properties of the fixed-central-character Plancherel theorem, we'll need to briefly recall some facts about the construction of the Plancherel measure from \cite{Wal03} and \cite{AP05}.  Let $\OO$ be a tempered orbit in $\GL_2(L)^{\wedge, t}$ and let $\pi\in\OO$.  Then there is a parabolic subgroup $P$, a Levi subgroup $M\leq P$, and a discrete series representation $\omega$ of $M$ such that $\pi \cong I_P^G(\omega)$.

Let $X_u(M)$ be the group of unramified characters on $M(L)$, and let $\OO_M$ be the set $\{\omega\otimes\tau: \tau \in X_u(M)\}$.  Then there are surjections 
$$X_u(M) \twoheadrightarrow \OO_M \twoheadrightarrow \OO$$
where the first map is $\tau \mapsto \omega\otimes\tau$ and the second map is $\omega'\mapsto I_P^G\omega'$.

\begin{defn} \label{defn6.9} Let $i: X \to Y$ be a surjective, finite map of orbifolds equipped with measures $\mu_X,\, \mu_Y$. We say $i$ \emph{locally preserves measures} if there is an open $X'\subseteq X$ and an open cover $\{U_{\alpha}\}$ of $X'$ such that $\mu_X(X - X') = 0$, $\mu_Y(i(X - X')) = 0$, and for each $U \subseteq U_\alpha$, $\mu_X(U) = \mu_Y(i(U))$.
\end{defn}

This definition may be ugly, but has the following useful property: if $i: X \to Y$ locally preserves measures and $E \subseteq X$ is an open fundamental domain for the map (so that $E \to Y$ is injective and covers $Y$ up to a set of measure $0$), then for any function $h: Y \to \CC$ we have 
$$\int_Y f\, d\mu_Y = \int_E (f\circ i) \, d\mu_X.$$

We now define the \emph{canonical measure}.

\begin{defn} \label{defn6.10} Let $M\subseteq \GL_2(L)$ be a Levi subgroup with center $Z(M)$. Let $\OO$ be an orbit in $\GL_2(L)^{\wedge,t}$ induced from $M$. Consider the surjective, finite maps
$$X_u(Z(M)) \xleftarrow{i} X_u(M) \xrightarrow{j} \OO.$$
We give $X_u(Z(M))$ the Haar measure with total measure $1$. If measures $d\chi_M$ on $X_u(M)$ and $d\pi$ on $\OO$ are chosen so that $i,\, j$ locally preserve measures, then we call $d\pi$ the \emph{canonical measure} on $\OO$.\end{defn}

The Plancherel measure $\mupl$ is absolutely continuous with respect to the canonical measure $d\pi$: there is a continuous function $\nupl$ such that $d\mupl(\pi) = \nupl(\pi)\,d\pi$.  The Plancherel density function is given explicitly by 
\begin{align*} 
	\nupl(I_P^G\omega) 
	& = c(G|M)^{-2}\gamma(G|M)^{-1} \mu_{G|M}(\omega) d(\omega)
	\\ & = \gamma(G|M)^{-1}j(\omega)^{-1} d(\omega).
\end{align*}

The $\gamma$ and $c$ factors is as described and computed on p. 241 of loc. cit. The term $d(\omega)$ is the formal degree of $\omega$; this is defined by the condition that
$$\int_{A_M \bs M} \linf{\omega(m)v_1}{\tilde v_1}\linf{v_2}{\tilde \omega(m) \tilde v_2} \, dm = d(\omega)^{-1} \linf{v_1}{\tilde v_2} \linf{v_2}{\tilde v_1}$$
for $v_1,\, v_2 \in V_{\omega}$ and $\tilde v_1,\, \tilde v_2 \in V_{\tilde \omega}$, where $\tilde \omega$ is the contragredient representation.

The $j(\omega)$ is the scalar given by an intertwining operator $I_P^G\omega \to I_P^G\omega$; these intertwining operators are defined in ch. I of loc. cit.  Finally $\mu_{G|M}(\omega)$ is chosen to be equal to $c(G|M)^2 j(\omega)$.

\begin{rem} \label{rem6.11} Note the dependence on Haar measures: $\wh \phi(\pi)$ depends on a Haar measure on $G$, whereas $j(\omega)^{-1}$ and $d(\omega)$ depend inversely on Haar measures on $N,\, M$ respectively (where $N$ is the unipotent radical of $P = MN$); we choose Haar measures $dg,\, dm,\, dn,\, dk$ so that $dk$ is the restriction of $dg$ to the maximal compact subgroup $K$ and so that
$$\int_G \phi(g)\,dg = \int_{M}\int_N\int_K \phi(mnk)\,dk\,dn\,dm$$
for any $\phi\in C_c^\infty(G)$.
\end{rem}

We define similarly the fixed-central-character canonical measure.  Fix a character $\chi: L^\times \to \CC^\times$ and let $\OO_{\chi}$ be the subset of $\OO$ where $\chi_\pi = \chi$.  Fix $\omega$ such that $\pi = I_P^G\omega$ and $\chi_\pi = \chi$, and let $X_u(M)_0$ be the kernel of the restriction map $X_u(M) \to X_u(Z(G))$. Then the surjection $X_u(M) \to \OO$ restricts to a surjection $X_u(M)_0\to \OO_{\chi}$, and we define the canonical measure on $\OO_{\chi}$ so that
\begin{itemize} 
	\item There is a Haar measure on $X_u(M)_0$ such that the map $X_u(M)_0\to \OO_{\chi}$ locally preserves measures, and
	\item $\OO$ and $\OO_{\chi}$ have the same canonical measure.
\end{itemize}

We now discuss the fixed-central-character Plancherel measure, and list some of its properties.

\begin{prop} \label{prop6.12} Let $L$ be a local field, let $\chi: L^\times \to \CC$ be a character.  Let $\GL_2(L)^{\wedge,t,\chi}$ be the subset of the tempered spectrum consisting of those representations $\pi$ with $\chi_\pi = \chi$.  There is a unique measure $\mupl_{\chi}$ on $\GL_2(L)^{\wedge, t, \chi}$ such that, for any $\phi\in \HH(\GL_2(L),\, Z(L),\, \chi)$ we have
$$\phi(1) = \int_{\GL_2(L)^{\wedge,t,\chi}} \wh \phi(\pi)\,d\mupl_{\chi}(\pi).$$
Moreover, the fixed-central-character Plancherel measure satisfies the following properties:
\begin{enumerate}[(i)]
	\item Let $d\pi$ be the canonical measure on $\GL_2(L)^{\wedge,t}$, let $\nupl$ be the Plancherel density function with respect to the $d\pi$, and let $d\pi_{\chi}$ be the canonical measure on $\GL_2(L)^{\wedge,t,\chi}$.  Then $d\mupl_{\chi} = \nupl d\pi_{\chi}$.
	\item Let $\pi \in \GL_2(L)^{\wedge,t,\chi}$.  Then $\mupl_{\chi}(\pi) \neq 0$ if and only if $\pi$ is a discrete series representation.  In this case, $\mupl_{\chi}(\pi) = d(\pi)$.
	\item Sauvageot's density theorem holds for the fixed-central-character Plancherel measure, in the following sense: given $\wh f_{\chi}$ on $\GL_2(L)^{\wedge,\chi}$ that is supported on a finite set of Bernstein components and that is continuous outside a set of Plancherel measure zero, we may find $\phi,\,\psi \in \HH(\GL_2(L),\, Z,\, \chi)$ such that $|\wh f_{\chi}(\pi) - \wh \phi(\pi)| \leq \wh\psi(\pi)$, and such that $\mupl_{\chi}(\wh\psi) < \epsilon$.
\end{enumerate}
\end{prop}

Property (iii) will be necessary to extend the methods of \cite{FL14}, \cite{Shi12}, and \cite{ST12} to the fixed-central-character setting.  Properties (i) and (ii) will be necessary to apply the computations of \cite{CMS90} and \cite{AP05} to the fixed-central-character setting.

\subsection{Explicit Computation of the Fixed Central Character Plancherel Measure}
\label{subsec:6.3}

In this section, we use the results of \cite{CMS90} and \cite{AP05} to determine explicitly the fixed-central-character Plancherel measure for $\GL_2(L)$, where $L$ is a local field.  Throughout this section, $q$ is the cardinality of the residue field of $L$.

\begin{comp}
\label{comp6.13}
Parts (1), (2), (3), and part of (4) have been computed by \cite{Shi12}, and we recall the results here, with appropriate citations in \cite{AP05}, and give an explicit value of the formal degree for supercuspidal representations, as computed in \cite{CMS90}. Note that Aubert-Plymen's function $\mu_{G|M}(\omega)$ satisfies
$$\mu_{G|M}(\omega)c(G|M)^{-2} \gamma(G|M)^{-1} = \gamma(G|M) j(\omega)^{-1}.$$
Here $c(G|M)$ is defined as in Waldspurger directly following the definition of $\gamma(G|M)$; for $G = \GL_2(L)$, we have $c(G|M) = 1$ for all Levi subgroups $M$.
\begin{enumerate}[(1)]
	 \item If $\OO$ corresponds to $\chi_0 \times \chi_0$, then $\omega(\OO) = \frac{1}{2}$. We have $M = T$ so $\gamma(G|M) = \frac{q+1}{q}$, and $d(\omega) = 1$. \cite[Theorem. 4.4]{AP05} then gives
$$\mu_{G|M}(\chi\otimes \chi') = \frac{(q + 1)^2}{q^2} \bigg|\frac{1 - (\chi'\chi^{-1})(\varpi)}{1 - q^{-1}(\chi'\chi^{-1})(\varpi)}\bigg|^2$$
so that
$$\nu^{\pl}(\chi\otimes \chi') = \frac{q+1}{q}\left|\frac{1 - (\chi'\chi^{-1})(\varpi)}{1 - q^{-1}(\chi'\chi^{-1})(\varpi)}\right|^2.$$
Integrating this function on $S^1 \times S^1$ yields $2$, so that $\mupl(\OO) = 1$.

We remark that the density function is independent of choice of uniformizer since $\chi_2$ and $\chi_1$ differ by an unramified character.

If we fix the central character, note first that we must have $\chi_0^2 = \chi^2|_{\oo_L^\times}$.  The canonical measure of $\OO_{\chi}$ is still 1/2 and the Plancherel density function still integrates to $1$.
	\item If $\OO$ is a principal series orbit corresponding to $\chi_0 \times \chi_0'$, then the canonical measure is $1$. \cite[Theorem 4.3]{AP05} says that $\nu^{\pl}$ is constant on such orbits and equal to 
$$\gamma(G|M) q^{c(\chi_0^{-1}\chi_0')} = \frac{q + 1}{q} q^{c(\chi_0^{-1}\chi_0')}$$
so that $\mupl(\OO) = \frac{q + 1}{q} q^{c(\chi_0^{-1}\chi_0')}$.

	If we fix a central character $\chi$ where $\OO_{\chi}\neq \emptyset$, then $\OO_{\chi}$ is topologically isomorphic to $S^1$.  The canonical measure of $\OO$ is the Haar measure of $S^1$, and the Plancherel measure has uniform density $\frac{q+1}{q} q^{c(\chi_0^{-1}\chi_0'}$.

	\item If $\OO$ is a Steinberg orbit, then $M = G$ and thus the canonical measure is $2$. The $\gamma$ and $j$-terms are uniformly $1$, so we simply need to find the formal degree. The formal degree of a Steinberg representation of $\GL_2(L)$ is $\frac{q-1}{2}$, so $\nu^{\pl} = \frac{q-1}{2}$ and $\mupl(\OO) = q-1$; see \cite[(17)]{AP05} or \cite[(2.2.2)]{CMS90} and note that the formal degree of a Steinberg is invariant under twisting by \emph{any} unitary character.
	
	If we fix a central character then $\OO_{\chi}$ consists of two disjoint points, each of measure $d(\pi) = \frac{q-1}{2}$.
	\item If $\OO$ is a supercuspidal orbit, then the same logic as above says that $\nu^{pl}(\pi) = d(\pi)$ and that this is constant on $\OO$, so that $\mupl(\OO) = \frac{2 d(\pi)}{r(\pi)}$.
	
	Let $\pi = \pi_{\eta}$ for $\eta: L'^\times \to \CC^\times$, with conductor $c(\eta)$. By fact (iv) after Definition 2.3, if $L'/L$ is unramified then $r(\pi) = 2$, and if $L'/L$ is ramified then $r(\pi) = 1$. Moreover, the formal degrees are computed in \cite[Theorem 2.2.8]{CMS90}. From the remark between (2.1.2) and (2.1.3) of loc. cit., we deduce that $\alpha(\eta)$ is the minimal conductor of all characters of the form $\eta \cdot (\chi \circ N_{L'/L})$, as $\chi$ ranges over all characters of $L^{\times}$.
	
Then Theorem 2.2.8 of loc. cit. proves that if $L'/L$ is unramified, then $d(\pi) = (q - 1)q^{\alpha(\eta) - 1}$ (note that the quantity given in 2.2.8 must be multiplied by $\frac{q-1}{2}$ because they choose their Haar measure so that $\vol(KZ/Z) = d(\St) = \frac{q-1}{2}$, whereas we choose it to be $1$, and the formal degree depends inversely on the choice of Haar measure). Similar logic says that if $L'/L$ is ramified, then $d(\pi) = \frac{1}{2}(q^2 - 1) q^{\frac{\alpha(\eta)}{2} - 1}$.  If $L'/L$ is ramified and $\eta$ is trivial on $1 + \pp_{L'}^{2r + 1}$ then we can pick $\chi: L^\times \to \CC^\times$ such that $\eta\cdot (\chi\circ N_{L'/L})$ is trivial on $1 + \pp_{L'}^{2r}$; therefore, $\alpha(n)$ is even in this case.

In either case, if we fix a central character $\chi$, then each supercuspidal representation of central character $\chi$ satisfies $\mupl_{\chi}(\pi) = d(\pi)$.
\end{enumerate}
\end{comp}

\begin{rem} \label{rem6.14} It is worth comparing the tempered orbits of our situation to Weinstein's \emph{inertial types} at finite places \cite{Wei09}.  Using our characterisation of orbits $\OO$, we see that if $\pi$ and $\pi'$ are tempered representations in the same orbit, then their associated Weil-Deligne representations $\rho(\pi),\, \rho'(\pi)$ have the same restriction to the inertia subgroup $I_L$ and the same monodromy operator. As such, two tempered representations are in the same orbit if and only if they have the same inertial type.

We claim that if an inertial type $\tau^\infty = (\tau_\pp)_{\pp\nmid \infty}$ is unramified outside the finite set $S$ of finite places, and $\tau_\pp$ corresponds to the orbit $\OO_\pp$, then 
$$d(\tau^{\infty}) = \mupl_S\left(\prod_{\pp\in S} \OO_\pp\right).$$
When $\OO$ is non-supercuspidal, this follows simply by comparing $d(\tau_\pp)$ in \cite[pp. 1390, 1393]{Wei09} to $\mupl_\pp(\OO_{\pp})$ as given in Computation \ref{comp6.13}.

There is a discrepancy when $\OO$ is a supercuspidal orbit corresponding to a character $\eta$ on a ramified extension. We believe this to be a minor miscomputation. The value of $\dim \tau(\pi)$ is given on page 1394 and should be equal to $|\GL_2(\OO_F):J^0|$, where $J^0$ is given as in 
(3) on page 1398. An explicit computation of $J^0$ shows that the index is $(q^2 - 1)q^{\frac{c(\eta) - 2}{2}}$, not $(q^2 - 1)q^{c(\eta) - 2}$ as on page 1394. This matches up with the Plancherel measure of the supercuspidal orbit as given in Computation \ref{comp6.13} (4), following \cite{CMS90}.
\end{rem}

\section{Counting Measures and Test Functions}
\label{sec:7}
In this section, we switch back to the global setting. We'll adapt the counting measure of (9.4) of \cite{ST12} to our setting. Throughout this section we fix
\begin{itemize}
	\item a totally real field $F$,
	\item an irreducible, finite-dimensional representation $\xi$ of $\GL_2(F_\infty)$ with discrete-series complement $\pi_{\xi}$ and Clozel-Delorme function $\phi_{\xi}$ (see Corollary \ref{cor5.4})
	\item an automorphic character $\chi: \AAA^\times \to \CC^\times$ extending $\chi_{\xi}$,
	\item a finite set $S$ of finite places. We set $F_S = \prod_{v\in S} F_v$, so that $\GL_2(F_S) = \prod_{v\in S} \GL_2(F_v)$.
\end{itemize}

We also fix the following notation:
\begin{itemize}
	\item $K_v = \GL_2(\oo_{F,v})$ for any finite place $v$,
	\item $\phi_S \in \HH(\GL_2(F_S),\, Z(F_S),\, \chi_S)$ with Plancherel transform $\widehat \phi_S$ on $\GL_2(F_S)^{\wedge,\chi_S}$,
	\item $\widehat f_S,\, \widehat h_S$ denote elements of $\mathscr{F}_0(\GL_2(F_S))^{\wedge}$ (or $\mathscr{F}_0(\GL_2(F_S)^{\wedge,\chi})$), and
	\item $\phi^{S,\,\infty}$ is a product of smooth functions $\phi_v \in \HH(\GL_2(F_v),\, F_v^\times,\, \chi_v)$ for finite places $v\not \in S$. We will assume that $\phi_v = \one_{F_v^\times K_v}$ at all but finitely many places, and such that $\phi_v$ is supported on $F_v^\times K_v$ everywhere.
\end{itemize}

We will begin with a definition:

\begin{defn}\label{defn7.1}
Fix $S$ and $\chi$ as above. Given a tuple $(\widehat f_S,\, \widehat \phi^{S,\infty},\, \xi)$, we define a multiset 
$$\FF= \FF_{\disc,\chi}(\widehat f_S,\, \widehat \phi^{S,\infty},\, \xi)$$ 
as follows: for a discrete automorphic representation $\pi = \pi_S \otimes \pi^{S,\, \infty} \otimes \pi_\infty$ with $\chi_\pi = \chi$, $\pi$ occurs in $\FF$ with multiplicity
$$a_{\FF}(\pi) = (-1)^{[F:\QQ]}m_{\disc}(\pi)\cdot \widehat f_S(\pi_S)\cdot \widehat \phi^{S,\infty}(\pi^{S,\infty})\cdot\tr_{Z(F_\infty)} \pi (\phi_{\xi}).$$
Here $m_{\disc}(\pi)$ is the multiplicity of $\pi$ in the discrete spectrum of $\GL_2(\AAA)$.

Define $\FF_{\cusp,\chi}$ similarly, but with $m_{\disc}$ replaced by $m_{\cusp}$, the multiplicity in the cuspidal spectrum.
\end{defn}

It follows from Harish-Chandra's finiteness theorem that, for $\FF$ as defined above, $a_{\FF}(\pi) = 0$ for all but finitely many $\pi$.  Moreover, $m_{\disc}(\pi) = 0$ or $1$ by strong multiplicity one.  Also, residual spectrum of $\GL_2(\AAA)$ consists of one-dimensional representations, so if $\dim \xi > 1$ then $\FF_{\cusp} = \FF_{\disc}$ as a multiset.

\begin{defn} \label{defn7.2} Given a multiset $\FF$, say $\pi \in \FF$ if $a_{\FF}(\pi) \neq 0$. If $\FF$ is finite, we define 
$$|\FF| = \sum_{\pi} a_{\FF}(\pi).$$
\end{defn}

\begin{rem}\label{defn7.3} We have borrowed the multiset notation from \cite{Shi12} and \cite{ST12}, but we have both simplified and generalized to match our needs. For instance, we have eliminated their set $S_1$ (or rather, assumed $S_1$ is empty) and let $S = S_0$. On the other hand, we have generalized their insistence that $\phi^{S,\infty}$ be an idempotent element corresponding to an open-compact subgroup; this will slightly simplify our proof, and will be strictly necessary when we show a partial extension of our result to newforms. We have also restricted to an arbitrary fixed central character.
\end{rem}

\begin{defn}\label{defn7.4} Fix an irreducible finite dimensional representation $\xi$ of $\GL_2(F_\infty)$, an automorphic character $\chi$ extending $\chi_\xi$, and $\phi^{S,\infty}\in \HH(\GL_2(\AAA^{S,\infty}),\, Z(\AAA^{S,\infty}),\, \chi^{S,\infty})$. We define the \emph{counting measures} $\mucusp_{\phi^{S,\infty},\xi,\chi}$ and $\mudisc_{\phi^{S,\infty},\xi,\chi}$ as linear functionals on $\mathscr{F}_0(\GL_2(F_S)^{\wedge,\chi})$ by 
$$\mucusp_{\phi^{S,\infty},\xi,\chi}(\widehat f_S) 
	= \frac
		{|\FF_{\cusp,\,\chi}(\widehat f_S,\,\widehat\phi^{S,\infty},\,\xi)|}
		{\tau_Z(G)\cdot\phi^{S,\infty}(1)\cdot\dim \xi}$$
and
$$\mudisc_{\phi^{S,\infty},\xi,\chi}(\widehat f_S) 
	= \frac
		{|\FF_{\disc,\,\chi}(\widehat f_S,\,\widehat\phi^{S,\infty},\,\xi)|}
		{\tau_Z(G)\cdot\phi^{S,\infty}(1)\cdot\dim \xi}$$
Here $\tau_Z(G)$ is the measure of $\GL_2(F)Z(\AAA)\bs \GL_2(\AAA)$, computed using the Euler-Poincar\'e measure at $\infty$ and the canonical measure at all finite places.
\end{defn}	
 
\subsection{Test Functions for Counting Cusp Forms}\label{subsec7.4}

We begin by defining the test functions we'll use to count cusp forms:

\begin{defn}\label{defn7.5} Let $\chi$ be an automorphic character with conductor $\ff(\chi)$ and let $\nn$ be a nonzero ideal in $\oo_F$ with $\ff(\chi)\mid \nn$. We define $\phi_{\nn, \chi}\in \HH(\GL_2(\AAA^\infty),\, Z(\AAA^\infty),\, \chi^\infty)$ as a product of local factors, as follows
\begin{itemize} 
	\item At all places $\pp$ not dividing $\nn$, $\phi_{\nn,\chi, \pp}$ is supported on $F_{\pp}^{\times}K_{\pp}$, with $\phi(z\cdot K_{\pp}) = \chi_{\pp}^{-1}(z)$. 
	\item Otherwise, if $\ord_{\pp}(\nn) = r$, then $\phi_{\nn,\chi,\pp}$ is supported on $F_{\pp}^\times \Gamma_0(\pp^r)$, and
	$$\phi_{\nn,\chi,\pp}\twomat{a}{b}{c}{d} = \vol(\Gamma_0(\pp^r))^{-1}\chi_{\pp}^{-1}(a).$$
\end{itemize}
\end{defn} 

\begin{lem}\label{lem7.6} 
Let $\pi_{\pp}$ have central character $\chi_{\pp}$ and let $\ord_{\pp}(\nn) = r$. Then $\tr \pi(\phi_{\nn,\,\chi,\,\pp})$ is the dimension of the space of vectors $v\in V_{\pp}$ such that $\pi(\gamma)v = \chi(a)v$ for any $\gamma = \smallmat{a}{b}{c}{d} \in \Gamma_0(\pp^r)$.
 \end{lem}
\begin{proof} Let
$$\phi_0(g) = \begin{cases}
	\phi_{\nn,\chi,\pp}(g) & |\det(g)| = 1
	\\ 0 & \text{otherwise}
\end{cases}$$
so that $\phi_{\nn,\,\chi,\,\pp}$ is the average of $\phi_0$ with respect to $\chi_{\pp}$. As such, for any $\pi_{\pp}$ with $\chi_{\pi_{\pp}} = \chi_{\pp}$, we have $\tr_{Z(F_\pp)} \pi(\phi_{\nn,\chi,\pp}) = \tr \pi(\phi_0)$.

On the other hand, it is elementary to check that $\pi(\phi_0)$ is a projection from $V_{\pi}$ onto the space of vectors $v$ so that $\pi \smallmat{a}{b}{c}{d} v = \chi_{\pp}(a)\cdot v$ for $\smallmat{a}{b}{c}{d} \in \Gamma_0(\pp^r)$. This completes the proof.
\end{proof}

\begin{prop}\label{prop7.7} Fix the following data:	
	\begin{itemize}
		\item A finite set $S$ of finite places;
		\item an irreducible finite-dimensional representation $\xi$ of $\GL_2(F_\infty)$, with complementary discrete series representation $\pi_\xi$;
		\item an automorphic character $\chi$ of conductor $\ff$ extending $\chi_\xi$;
		\item a nonzero ideal $\nn$ of $\oo_F$ wih $\ff\mid \nn$. Write $\nn = \nn_S \nn^S$, where $\nn_S$ is divisible only by primes in $S$ and $\nn^S$ is coprime to $S$; and
		\item a function $\widehat h_S \in \mathscr{F}_0(\GL_2(F_S)^{\wedge})$.
	\end{itemize}
Let 
$$\FF = \FF_{\cusp,\chi}(\widehat h_S\cdot\widehat \phi_{\nn_S,\,\chi},\, \widehat \phi_{\nn^S,\chi},\, \xi).$$
Then $|\FF|$ counts the cuspidal $\GL_2(\AAA)$-representations with $\chi_\pi = \chi$, $\pi_{\infty} \cong \pi_{\xi}$, and conductor $\dd$ dividing $\nn$; such a representation $\pi$ is counted with multiplicity $\wh h_S(\pi_S) d(\nn/\dd)$.\end{prop}
\begin{proof}
Let $\pi$ be an irreducible cuspidal automorphic representation with central character $\chi$ and conductor $\dd$. Then it has a Whittaker model, so each of its archimedean components has a Whittaker model. If $\pi_v$ is generic and $\tr \pi_v(\phi_{\xi, v}) \neq 0$, then $\pi_v = \pi_{\xi_v}$, and $\tr \pi_i(\phi_{\xi, v}) = -1$ by Corollary \ref{cor5.4}.

By Lemma \ref{lem7.6} and the classical result of Casselman (see Theorem 4.24 and the discussion before Remark 4.25 of \cite{Gel75}), we have that $\tr \pi^{S,\infty}(\phi_{\nn^S,\chi}) = d(\nn^S/\dd^S)$. Similarly, $\wh h_S(\pi_S) \wh \phi_{\nn_S,,\chi}(\pi_S) = \wh h_S(\pi_S) d(\nn_S/\dd_S)$, completing the proof.
\end{proof}

\begin{cor}\label{cor7.8} Let $S,\, \wh h_S$ be as above. Let $k$ be a weight and $\chi$ a character of conductor $\ff$ occurring in weight $k$. Let $\xi_k = \bigotimes_{v\mid \infty} \xi_{k_v}$, where $\xi_{k_v} = \Sym^{k_v - 2}(\RR^2) |\det|^{\frac{-k_v - 2}{2}}$. If $\ff \mid \nn$ and $\nn_S,\, \nn^S$ are as above, and
$$\FF = \FF_{\cusp,\chi}(\widehat h_S\cdot\widehat \phi_{\nn_S,\,\chi},\, \widehat \phi_{\nn^S,\chi},\, \xi_k)$$
then $|\FF|$ counts the number of cusp forms of weight $k$, level $\nn$, and character $\chi$, where a cusp form $f$ is counted with multiplicity $\wh h_S(\pi_{f,S})$.
\end{cor}
\begin{proof} This follows directly from the previous proposition and the correspondence between cusp forms and cuspidal representations, once we note the following two facts:
\begin{itemize}
	\item If $f$ is a cusp form of weight $k$ and $\xi_k$ is as above, then $\pi_{f,\infty} = \pi_{\xi_k}$ \cite[Theorem 1.4]{RT11}; and
	\item If $f$ is an \emph{newform} of level $\dd$ and character $\chi$, then the multiplicity of $f$ in $S_k(\Gamma_1(\nn),\, \chi)$ is $d(\nn/\dd)$.
\end{itemize}
\end{proof}

\begin{cor}\label{cor7.9} 
Let $\xi,\, \chi,\, \, \ff,\, \nn$ be as above. 
Define $\phi^{\new}_{\nn,\chi} \in \HH(\GL_2(\AAA^\infty),\, Z(\AAA^\infty),\, \chi^\infty)$ by
$$\phi^{\new}_{\nn,\chi,\pp} = \begin{cases}
	\phi_{\nn,\chi,\pp} 
		& \ord_{\pp}(\nn/\ff) = 0
	\\ \phi_{\nn,\chi,\pp} - 2\cdot \phi_{\nn/\pp,\chi,\pp} 
		& \ord_{\pp}(\nn/\ff) = 1
	\\ \phi_{\nn,\chi,\pp} - 2\cdot \phi_{\nn/\pp,\chi,\pp} + \phi_{\nn/\pp^2,\chi,\pp}
		& \ord_{\pp}(\nn/\ff) \geq 2
\end{cases}.$$

Assume $\xi$ is a finite dimensional representation with $\chi_\xi = \chi_\infty$. If 
$$\FF = \FF_{\cusp,\chi} (\widehat h_S\widehat \phi^{\new}_{\nn_S,\,\chi},\, \widehat \phi^{\new}_{\nn^S,\chi},\, \xi)$$
then $|\FF|$ counts the number of automorphic representations $\pi$ of exact conductor $\nn$, $\chi_{\pi} = \chi$, and $\pi_{\infty} = \pi_{\xi}$, with multiplicity $a_{\FF}(\pi) = \widehat h_S(\pi_{S})$.

When $\xi = \xi_k$ is as in Corollary \ref{cor7.8}, $|\FF|$ counts the newforms of weight $k$, level $\nn$, and conductor $\chi$ with multiplicity $a_{\FF}(f) = \wh h_S(\pi_{f,S})$.
\end{cor}
\begin{proof}
The second statement follows from the first as in the proof of Corollary \ref{cor7.8}.

To prove the first statement, we just need to prove that $\widehat \phi^{\new}_{\nn,\chi}(\pi)$ is $1$ if $\ff(\pi) = \nn$, and zero otherwise. By writing the trace as a product of local traces, it's enough to show that if if $\ord_{\pp}(\nn) = r$, then we need to show that $\widehat \phi^{\new}_{\nn,\chi,\,\pp}(\pi_\pp)$ is $1$ if $c(\pi_\pp) = r$ and zero otherwise.

If $c(\pi_\pp) = r$, then $\widehat \phi_{\nn,\chi,\pp}(\pi_\pp) = 1$ and $\widehat \phi_{\nn/\pp,\chi,\pp}(\pi_\pp) = \widehat \phi_{\nn/\pp^2,\chi,\pp}(\pi_\pp) = 0$.

If $c(\pi_\pp) > r$ then evidently $\widehat \phi^{\new}_{\nn,\chi,\pp}(\pi_\pp) = 0$. If $c(\pi_\pp) = r' \leq r - 1$, then we have
\begin{align*} 
	\phi^{\new}_{\nn,\chi,\pp}(\pi_\pp) 
	& = \phi_{\nn,\chi,\pp}(\pi_\pp) - 2\cdot \phi_{\nn/\pp,\chi,\pp}(\pi_\pp) + \phi_{\nn/\pp^2,\chi,\pp}(\pi_\pp)
	\\ & = (r + 1 - r') - 2\cdot(r - r') + (r - 1 - r')
	\\ & = 0
\end{align*}
\end{proof}

Recall the counting measures defined in (\ref{defn7.4}). The goal of the next sections is to prove the following:

\begin{thm}[(Plancherel equidistribution theorem)] \label{thm7.10}Let $S$ be a finite set of places, $\xi$ an irreducible finite-dimensional $\GL_2(F_\infty)$-representation, $\chi$ an automorphic character with $\chi_{\infty} = \chi_\xi$, and $(\nn_{\lambda})$ a sequence of levels divisible by $\ff^S$ and coprime to $S$, such that $N(\nn_{\lambda}) \to \infty$. For simplicity let $\widehat \mu_{S, \lambda} = \widehat \mu_{S,\phi_{\nn,\chi},\xi}$, with superscript $\cusp$ or $\disc$. Then
$$\lim_{\lambda\to\infty} \mucusp_{S,\lambda}(\widehat f_S)
	= \lim_{\lambda\to\infty} \mudisc_{S,\lambda}(\widehat f_S)
	= \mupl_{S, \chi_S}(\widehat f_S).$$
	
\end{thm}

We conclude this section with a lemma, which will start the proof of the theorem.

\begin{lem} \label{lem7.11} In Theorem \ref{thm7.10}, the second equality implies the first.\end{lem}
\begin{proof} If $\dim \xi > 1$ then any discrete automorphic representation $\pi$ with $\pi_{\infty} = \pi_{\xi}$ is cuspidal, and so we are done in this case.

Otherwise, assume $\dim \xi = 1$. Fix $\widehat f_S$ and let $\widehat \one^t$ be the characteristic function of $\GL_2(F_S)^{\wedge,\, t}$ in $\GL_2(F_S)^{\wedge}$. The Plancherel measure is supported on the tempered spectrum, so $\mupl(\widehat f_S) = \mupl(\wh f_S\cdot \wh \one_t)$. For a positive function $\widehat h_S$, we have 
$$0 \leq 
	\mudisc_\lambda(\wh h_S) - \mucusp_\lambda(\wh h_S) 
	\leq \mudisc_\lambda(\wh h_S) - \mudisc_\lambda(\wh h_S \cdot \wh \one^t);$$
this follows because $\wh h_S$ is positive and because every discrete representation that is tempered at $S$ is cuspidal, since the residual spectrum of $\GL_2(\AAA)$ consists of one-dimensional representations. (It is conjectured that every cuspidal representation is, in fact, tempered everywhere; this is the generalized Ramanujan-Petersson conjecture). We therefore have
\begin{align*}
	|\mudisc_\lambda(\wh f_S) - \mucusp_\lambda(\wh f_S)|
	& \leq (\mudisc_\lambda - \mucusp_\lambda)(|\wh f_S|)
	\\ & \leq \mudisc_\lambda(|\wh f_S| - \wh \one^t\cdot |\wh f_S|)
\end{align*}
As $\lambda \to \infty$, the final term approaches $\mupl(|\wh f_S| - \wh \one^t\cdot |\wh f_S|)$. But the Plancherel measure is supported on the tempered spectrum, so this is zero, finishing the proof.
\end{proof}

\section{Asymptotic Bounds on Constant Terms and Orbital Integrals}
\label{sec:8}

The goal of this section is to bound the constant terms $Q_{\gamma}(\one_{(Z(\AAA^{\infty}\Gamma_0(\nn)})$ and orbital integrals $O_{\gamma}(\one_{Z(\AAA^{\infty})\Gamma_0(\nn)})$. We will begin by computing local orbital integrals and constant terms and then summarize the global consequences in subsection \ref{subsec:8.3}.
 
Throughout, we will choose measures on $G = \GL_2(F_\pp)$, $T$ the diagonal torus, $N$ the subgroup of upper-triangular unipotent matrices, and $K_{\pp} = \GL_2(\oo_{F,\pp})$ so that maximal compact subgroups are given measure $1$; this also ensures that $dg = dt\,dn\,dk$ under the Iwasawa decomposition $G = TNK$.

The key tool will be an analysis of the Bruhat-Tits tree for $\SL_2$. We recall a definition: 

\begin{defn}\label{defn8.1} Consider the $p$-adic field $F_{\pp}$. The \emph{Bruhat-Tits tree} $X$ of $\SL_2(F_{\pp})$ is a graph consisting of the following data:
	\begin{itemize}
	\item The set of vertices is the set of equivalence classes of rank-two lattices $\Lambda \subseteq F_{\pp}^2$, with $\Lambda \sim \Lambda'$ if they differ only by a scalar multiple.
	\item Two equivalence classes $[\Lambda],\, [\Lambda']$ are adjacent if and only if there are lattices $\Lambda\in [\Lambda],\, \Lambda' \in [\Lambda']$ such that $\Lambda \supsetneq \Lambda' \supsetneq \varpi\cdot \Lambda$. 
	\end{itemize}
\end{defn}

We briefly recall some facts: 
\begin{enumerate} 
	\item The degree of every vertex $v\in X$ is $q + 1$. To see this, fix a lattice $\Lambda$. If $\Lambda' \subset \Lambda$ with index $q$, then $\varpi \Lambda \subset \Lambda' \subset \Lambda$, and so $\Lambda'$ corresponds uniquely to a one-dimensional subspace in $\Lambda/\varpi \Lambda \cong \mathbb{F}_q^2$. On the other hand, if $\Lambda' \supset \Lambda$ with index $q$, then $\Lambda'$ is equivalent to $\varpi \Lambda'$, which is a sublattice of $\Lambda$ of index $q$. Moreover, if $\Lambda_1 \sim \Lambda$ then all index-$q$ sublattices of $\Lambda_1$ are equivalent to an index-$q$ sublattice of $\Lambda$.
	\item $X$ is a tree \cite[Theorem 1]{Ser80}.
\end{enumerate}

Let $\{e_1,\, e_2\}$ be the standard basis of $F_{\pp}^2$. $X$ has a distinguished line $A_0$ whose vertices correspond to the lattices with bases $\{e_1,\, \varpi^{i} e_2\}$; this is known as the \emph{standard apartment}. For fixed $g \in \GL_2(F_\pp)$, $A = g \cdot A_0$ is called an \emph{apartment}. Given a vertex $w$ and an apartment $A$, let $d(w,\, A)$ be the distance from $w$ to $A$. Because $X$ is a tree, there is a unique vertex $w'\in A$ such that $d(w,\, A) = d(w,\,w')$; we define $b_A(w) = w'$.

By the Iwasawa decomposition, every vertex has an associated lattice $\Lambda$ with basis $\{e_1,\, ae_1 + \varpi^s e_2\}$, where $s\in \ZZ$ and $a\in F_{\pp}$. We denote this vertex by $w_{a, s}$. Note that $w_{a, s} = w_{a', s'}$ if and only if $s = s'$ and $a - a'\in \oo_{F,\pp}$, and that $w_{a, s}\in A_0$ if and only if $a\in \oo_{F,\pp}$. It is elementary to check by induction that if $a \not\in \oo_{F,\pp}$ then $d(w_{a,\, s},\, A_0) = -v_P(a)$ and that $b_{A_0}(w_{a,s}) = w_{0, s - v(a)}$; this follows because $w_{a, s}$ is adjacent to $w_{\varpi\cdot a, s+1}$.

We say a set of vertices $\{w_0,\ldots,\, w_r\}$ is a \emph{segment} (of length $r$) if $d(w_i,\, w_{j}) = i - j$ for all $0 \leq i,\, j\leq r$.

The action of $\GL_2(F_\pp)$ on the set of lattices in $F_\pp^2$ descends to an action on $X$ by graph automorphisms. We have the following:

\begin{lem}\label{lem8.2} Let $\gamma \in \GL_2(F_\pp)$. Then $\gamma \in Z\cdot\Gamma_0(\pp^r)$ if and only if $\gamma$ fixes the length-$r$ segment $S_r = \{w_{0, 0}, \, w_{0, 1},\ldots,\, w_{0, r}\}$.

Moreover, $g^{-1}\gamma g \in Z\cdot \Gamma_0(\pp^r)$ if and only if $\gamma$ fixes $g\cdot S_r$.
\end{lem}
\begin{proof}
The second statement follows from the first. To prove the first, a quick computation yields that $\gamma$ fixes the lattice $\Lambda_i$ if and only if $\gamma \in Z\cdot \smallmat{1}{0}{0}{\varpi^i}^{-1} K_{\pp} \smallmat{1}{0}{0}{\varpi^i}$. The intersection of such subgroups from $i = 0$ to $i = r$ is $Z\cdot \Gamma_0(\pp^r)$.
\end{proof}

\subsection{Computation of Constant Terms}
\label{subsec:8.1}

Let $t = \smallmat{t_1}{0}{0}{t_2}$; we wish to compute the constant term $Q_t(\one_{Z\cdot\Gamma_0(\pp^r)})$. We begin with a lemma:

\begin{lem}\label{lem8.3} Let $t = \smallmat{t_1}{0}{0}{t_2} \in K_{\pp}$ and let $w\in X$ be a vertex. Then $t$ fixes $w$ if and only if $d(w,\, A_0) \leq v_{\pp}(t_1 - t_2)$.\end{lem}
\begin{proof} Write $w = w_{a, s}$, and note that $t$ fixes $w_{a, s}$ if and only if 
$$\twomat{t_1}{(t_1 - t_2)a}{0}{t_2} = \twomat{1}{a}{0}{\varpi^s}^{-1} \twomat{t_1}{0}{0}{t_2} \twomat{1}{a}{0}{\varpi^s} \in K\cdot Z$$
which occurs if and only if $(t_1 - t_2) a \in \oo_{F,\pp}$.

Since $d(w_{a, s},\, A_0) = -v_\pp(a)$, this completes the proof.
\end{proof}

\begin{prop}\label{prop8.4} Let $t_1 \neq t_2 \in \oo_{F,\pp}^{\times}$. Then 
$$Q_{\gamma}(\one_{Z \cdot \Gamma_0(\pp^r)}) \leq 
	\begin{cases}
	1 & r \leq v_{\pp}(t_1 - t_2) \\
	2 q^{v_{\pp}(t_1 - t_2)} \vol(\Gamma_0(\pp^r)) & r > v_{\pp}(t_1 - t_2).
	\end{cases}
$$
\end{prop}

\begin{proof} Fix a strictly upper-triangular matrix $n$ and note that for any $k\in K$ we can only have $k^{-1}tnk\in K$ if $n\in K$. Since $t_1 \neq t_2$, there is a $g$ so that $g^{-1}tng = t$ and therefore the set $X^{tn}$ of vectors fixed by $tn$ is of the form $g\cdot X^t$. If $A = g\cdot A_0$, then $w \in X^{tn}$ if and only if $d(w, \, A) \leq v_\pp(t_1 - t_2)$. 

Fix $n\in N \cap K$; we have $k^{-1}tnk \in Z \cdot \Gamma_0(\pp^r)$ if and only if the segment $k\cdot S_r \subset X^{tn}$. We note that the initial vertex of $k\cdot S_r$ is $w_{0,0}$; we will show that there number of such segments contained in $X^{tn}$ is at most $[K:\Gamma_0(\pp^r)]$ if $r \leq v(t_1 - t_2)$, and is at most $2q^{v_{\pp}(t_1 - t_2)}$ otherwise. The first statement is obvious simply by counting the total number of segments of length $r$ with a given initial point.

For the second case, we note the following: since $X$ is a tree, if $S = (w_0,\ldots,\, w_\ell)$ is a segment with $d(w_{1}, \,A) > d(w_{0},\, A)$, then $d(w_{i + 1}, \,A) > d(w_{i},\, A)$ for all $i$. As such, if $k \cdot S_r$ is a segment contained in $X^{tn}$, then for all $1 \leq i \leq r - v_{\pp}(t_1 - t_2)$, we have $d(w_i,\, A) \leq d(w_{i-1},\, A)$. As such, we claim that there are at most $2 q^{v_{\pp}(t_1 - t_2)}$ segments of the form $k\cdot S_r = \{w_0',\ldots,\, w_{r}'\}$ contained in $X^{tn}$. Because $k\in K$, we have $w_0' = w_0$. For each $1\leq i \leq r - v_{\pp}(t_1 - t_2)$, if $w_{i-1} \not \in A$, then $w_i$ is the unique neighbor of $w_{i-1}$ with $d(w_{i},\, A) < d(w_{i-1},\, A)$. If the $w_{i-1}\in A$ and $w_{i-2}\not\in A$, then $w_i$ must be one of the two neighbors of $w_{i - 1}$ in $A$. Finally, if $w_{i-1},\, w_{i-2}\in A$, then $w_{i}$ must be the other neighbor of $w_{i-1}$ in $A$. Finally, if $i > r - v_{\pp}(t_1 - t_2)$, then $w_{i}$ can be any of the $q$ neighbors of $w_{i-1}$ which are not equal to $w_{i-2}$. This completes the proof of the claim.

Therefore, for any $n \in N\cap K$ we have
$$\int_{K} \one_{Z\cdot \Gamma_0(\pp^r)}(k^{-1}tnk)\, dk \leq 	
\begin{cases}
	1 & r \leq v_{\pp}(t_1 - t_2) \\
	2 q^{v_{\pp}(t_1 - t_2)} \vol(\Gamma_0(\pp^r)) & r > v_{\pp}(t_1 - t_2).
\end{cases}$$
and so integrating over $n \in N \cap K$ completes the proof.
\end{proof} 

We will also need to compute the constant term $Q_{z}(\one_{Z\cdot \Gamma_0(\pp^r)})$ for a central element $z$.

\begin{prop}\label{prop8.5} Let $z \in Z(F_\pp)$. Then 
	$$Q_{z}(\one_{Z\cdot \Gamma_0(\pp^r)}) = \begin{cases} 
		\frac{2}{q+1} q^{-k} & r = 2k + 1
		\\ q^{-k} & r = 2k.
	\end{cases}$$
In particular, $Q_z(\one_{Z \cdot \Gamma_0(\pp^r)}) \leq q^{-r/2}$.
\end{prop}
\begin{proof} We can assume that $z = 1$ and once again find the fixed subspace $X^n$ for $n = \smallmat{1}{b}{0}{1} \in K$. Let $w_{a,\, s}$ be as in the beginning of the section. Since
$$\twomat{1}{a}{0}{\varpi^s}^{-1} \twomat{1}{b}{0}{1} \twomat{1}{a}{0}{\varpi^s} = \twomat{1}{b\varpi^{s}}{0}{1}$$
we see that $w_{a,s} \in X^n$ if and only if $s \geq - v(b)$. In particular, if $b_{A_0}(w) = w_{0,s}$, then $d(w,\, w_{0,s}) \leq s + v(b)$. Alternatively, $X^n$ is the union of balls of radius $s + v(b)$ around $w_{0,s}\in A_0$, for $s \geq - v(b)$.

For fixed $n$, the volume of the set 
$$\{k\in K: k^{-1}nk\in Z \cdot \Gamma_0(\pp^r)\}$$
is the product of $\vol(\Gamma_0(\pp^r))$ with the number of segments $\{w_0,\ldots,\, w_r\}$ whose basepoint is $w_0 = w_{0,0}$ and which are contained in $X^n$. Let $n = \smallmat{1}{b}{0}{1}$. If $v_{\pp}(b) \geq r$ then all length-$r$ segments with basepoint $w_0$ are contained in $X^n$, so the total volume is $1$. If $r > v_{\pp}(b)$, then for any $i \leq \lceil \frac{r - v(b)}{2}\rceil$ we must have $w_i = w_{0,\, i}$; for each subsequent step there are $q$ choices, so the total number of segments contained in $X^n$ is $q^{\lfloor \frac{v(b)}{2}\rfloor}$. 

As such, we compute
\begin{align*} 
	\int_{N} \int_K \one_{Z\cdot\Gamma_0(\pp^r)}(k^{-1}nk) \,dk\,dn 
	& = q^{-r} + \frac{1}{q^{r - 1}(q+1)} \sum_{j = 0}^{r - 1} (q-1) q^{-j - 1} q^{\lfloor j/2\rfloor}.
\end{align*}
An elementary computation using induction shows that this is equal to the quantity stated.
\end{proof}

\subsection{Computation of Orbital Integrals}\label{subsec:8.2}
The goal of this section is to prove:

\begin{prop} \label{prop8.6} Let $\gamma$ be a non central, semisimple element of $\GL_2(F_\pp)$. Then 
$$O_{\gamma}(\one_{Z \cdot \Gamma_0(\pp^r)}) \leq 2\cdot \vol(\Gamma_0(\pp^r))\cdot O_{\gamma}(\one_{Z\cdot K})^2.$$
\end{prop}

We'll break this into two cases: the case where $\gamma$ is elliptic, and the case where $\gamma$ is non-elliptic.

\begin{lem} \label{lem8.7} If $\gamma$ is elliptic and noncentral then the set $X^{\gamma}$ is finite.\end{lem}
\begin{proof} We can compute the fixed set directly, assuming $\gamma \in K$ by conjugating and multiplying by an element of the center. If $\gamma$ is elliptic then it is conjugate to a matrix of the form 
$$\twomat{x}{y}{\alpha y}{x}$$
where $\alpha$ is either a unit that is not a square, or $\alpha$ is a uniformizer.

If $\alpha$ is a unit, then the $X^{\gamma}$ is the single point $\{w_{0,0}\}$. If $\alpha$ is a uniformizer, then $X^{\gamma}$ consists of those vertices $w$ with $d(w,\, S_1) \leq v(y)$, where $S_1$ is the length-one segment $\{w_{0, 0}, \, w_{0,\,1}\}$. In either case, $X^\gamma$ is finite.\end{proof}

We will now prove Proposition \ref{prop8.6}.

\begin{proof}[Proof of Proposition \ref{prop8.6}] Assume first that $\gamma$ is elliptic, and by conjugating assume $\gamma \in \Gamma_0(\pp^r)$. Then $O_{\gamma}(\one_{Z\cdot K_\pp})$ is the cardinality of $X^{\gamma}$. As such, for a given length $r$, there are at most $O_{\gamma}(\one_{Z \cdot K_\pp})^2$ segments of length $r$ contained in $X^{\gamma}$ since each segment is determined uniquely by its two endpoints. For a given segment $S'_r$, the volume of the set $\{g\in G_{\gamma}\bs \GL_2(F_\pp): g\cdot S_r = S'_r\}$ is $\vol(\Gamma_0(\pp^r))$. This finishes the proof when $\gamma$ is elliptic.

If $\gamma$ is diagonalizable, we can assume $\gamma = \twomat{t_1}{0}{0}{t_2}\in K$. In this case, \cite[Lemma 9]{van72} tells us that 
$$O_{\gamma}(\one_{Z\cdot \Gamma_0(\pp^r)}) = |D^{G}_T(\gamma)|^{-1/2}_{\pp}Q_{\gamma}(\one_{Z\cdot \Gamma_0(\pp^r)})$$
where $D^{G}_M(\gamma)$ is the determinant of $1 - \Ad(\gamma)$ acting on $\Lie(G)/\Lie(T)$. In our situation we have 
$$|D_T^G(\gamma)| =\left|\left(1 - \frac{t_1}{t_2}\right)\left(1 - \frac{t_2}{t_1}\right)\right| = |t_1 - t_2|^2.$$

First, this lemma and Proposition \ref{prop8.4} prove that $O_{\gamma}(\one_{Z\cdot K_\pp}) = |t_1 - t_2|_{\pp}^{-1} = q^{v(t_1 - t_2)}$. Applying these results to $\one_{Z \cdot \Gamma(\pp^r)}$ gives
$$O_{\gamma}(\one_{Z\cdot \Gamma_0(\pp^r)}) \leq 2\cdot O_{\gamma}(\one_{Z\cdot K})^2 \cdot \vol(\Gamma_0(\pp^r))$$
completing the proof.
\end{proof}

\subsection{Summary of global consequences}\label{subsec:8.3}
We summarize the global consequences for use in subsequent sections below:

\begin{prop}\label{prop8.8} Let $\gamma \in \GL_2(F)$ be semisimple and let $\nn\subseteq \oo_{F}$ be an ideal. Then
\begin{enumerate}
	\item If $\gamma \in Z(F)$, then 
		$$Q_{\gamma}(\one_{Z(\AAA^\infty) \Gamma_0(\nn)}) \leq N(\nn)^{-1/2}$$
	\item If $\gamma =\smallmat{t_1}{0}{0}{t_2} \in T(F) - Z(F)$, then 
		$$Q_{\gamma}(\one_{Z(\AAA^\infty) \Gamma_0(\nn)}) \leq |N_{F/\QQ}(t_1 - t_2)|_{\RR}\cdot 2^{P(\nn)}\cdot N(\nn)^{-1}$$
	where $P(\nn)$ is the number of primes dividing $\nn$.
	\item If $\gamma \in \GL_2(F) - Z(F)$ is semisimple, then
		$$O_{\gamma}(\one_{Z(\AAA^\infty) \Gamma_0(\nn)}) \leq O_{\gamma}(K^\infty)^2 \cdot 2^{P(\nn)}\cdot N(\nn)^{-1}.$$
\end{enumerate}
\end{prop}
\begin{proof} This follows from Propositions \ref{prop8.4}, \ref{prop8.5}, and \ref{prop8.6} upon decomposing the orbital integrals and constant terms as a product of local orbital integrals and constant terms.\end{proof}

Because $2^{P(\nn)} \cdot N(\nn)^{-1}$ decreases as $o(N(\nn)^{-1 + \epsilon})$ for every $\epsilon > 0$, we have the following
\begin{cor} \label{cor8.9} For every semisimple, noncentral $\gamma\in \GL_2(F)$ and every $\epsilon > 0$, there is a $C_{\epsilon,\gamma} > 0$ such that 
$$Q_{\gamma}(\one_{Z(\AAA^\infty) \Gamma_0(\nn)}),\, O_{\gamma}(\one_{Z(\AAA^\infty) \Gamma_0(\nn)}) < C_{\epsilon,\gamma} N(\nn)^{-1 + \epsilon}$$
for all ideas $\nn \subseteq \oo_{F,\pp}$.
\end{cor}


\section{The Plancherel Equidistribution Theorem}
\label{sec:9}
In this section, we use the results of the previous section to prove our key intermediate result:

\begin{thm}[(Plancherel equidistribution theorem)] \label{thm9.1} Fix a finite set of places $S$. Let $\xi$ be a finite-dimensional $\GL_2(F_\infty)$-representation, let $\chi$ be a character of conductor $\ff$ with $\chi_\infty = \chi_{\xi}$, and let $\widehat f_S \in \mathscr{F}_0(\GL_2(F_S)^{\wedge,\chi})$. Let $(\nn_\lambda) \to \infty$ be a sequence of levels coprime to $S$ with $\ff^S \mid \nn_\lambda$ and $N(\nn_{\lambda}) \to \infty$. Then
	$$\lim_{\lambda\to\infty} \mucusp_{\phi_{\nn_\lambda,\chi} \, \xi,\, \chi}(\widehat f_S) 
	= \lim_{\lambda\to\infty} \mudisc_{\phi_{\nn_\lambda,\chi} \, \xi,\, \chi}(\widehat f_S) 
	= \mupl_\chi(\widehat f_S).$$
\end{thm}

Before the proof, we'll need a lemma:
\begin{lem}\label{lem9.2} Fix a compact set $C_S$ of $\GL_2(F_S)/Z(F_S)$. Then there are only finitely semisimple conjugacy classes $\{\gamma\} \in \GL_2(F)/Z(F)$ such that $\{\gamma^\infty\}$ intersects $C_SK^{S,\infty}$, and such that $\gamma$ is elliptic at all infinite places.\end{lem}
\begin{proof} First, because $|\det \gamma|_\pp = 1$ for all $\pp \not \in S$, and $|\det \gamma|_{S}$ can be chosen to lie in the finite set $I_S/P_S^2$ (where $I_S$ is the group of ideals divisible only by primes in $S$, and $P_S$ is the subgroup of principal ideals), then $|\det \gamma|_S$ can be chosen in a finite set. Because $\oo_F^{\times}/(\oo_F^\times)^2$ is finite, we can actually assume that $\det \gamma$ lies in a finite set by shifting by an element of $Z(F)$.

Because $\{\gamma\}$ intersects $C_SK^{S,\infty}$, its trace lies in some fractional ideal $\fa$ in $F$. Let $\fa_\infty$ be image of $\fa$ under $F\into \RR^{n}$. Fix a determinant $D\in F^\times$. If $\gamma$ is elliptic at each infinite place we must have $\tr(\gamma)^2_v \leq 4D_v$ for each infinite place, so $\tr(\gamma)_\infty$ lies in some compact set. Since $\fa_\infty$ is a lattice, then there are at most finitely many traces $\gamma$ can take for each determinant. A semisimple conjugacy class is determined by its trace and determinant, completing the proof.
\end{proof}

We now prove the theorem.

\begin{proof}[Proof of Plancherel equidistribution theorem]
For simplicity we write 
$$\phi_{\lambda} = \phi_{\nn_{\lambda,\,\chi}} \in \HH(\GL_2(\AAA^{S,\infty}),\, Z(\AAA^{S,\infty}),\,\chi^{S,\infty})$$ and 
$$\mudisc_\lambda = \mudisc_{\phi_{\nn_{\lambda}},\xi,\chi}.$$

Let's first assume that $\widehat f_S = \widehat \phi_S$ for some $\phi_S \in \HH(\GL_2(F_S),\, Z(F_S),\, \chi_S)$. In this case, we have 
\begin{align*} 
	I_{\spec}(Z(\AAA),\, \chi,\, \phi_S\phi_{\lambda}\phi_{\xi}) 
	& = \sum_{\pi} (\tr \pi_{S}(\phi_S)) \cdot (\tr \pi^{S,\infty}(\phi_\lambda)) \cdot \tr( \pi_\infty(\phi_{\xi}))
	\\ & =(-1)^{[F:\QQ]}|\FF_{\disc,\chi}(\widehat \phi_S,\, \widehat \phi_{\lambda},\, \xi)|
 \end{align*}
where in each sum, $\pi$ runs over the discrete automorphic representations of $\GL_2(\AAA)$ with central character $\chi$.

As such, we have 
\begin{align*} 
	\mudisc_{\lambda}(\widehat \phi_S) & =
		 (-1)^{[F:\QQ]}\frac
		 	{I_{\spec}(Z(\AAA),\, \chi,\, \phi_S\cdot \phi_{\lambda}\cdot \phi_{\xi})}
		 	{\tau_Z(G)\cdot \phi_{\lambda} (1)\cdot \dim(\xi)}
	\\ & = 
		 (-1)^{[F:\QQ]}\frac
		 	{I_{\geom}(Z(\AAA),\, \chi,\, \phi_S\cdot \phi_{\lambda}\cdot \phi_{\xi})}
		 	{\tau_Z(G)\cdot \phi_{\lambda}(1)\cdot \dim(\xi)}
\end{align*}
Recall from the comment after \ref{prop6.3} that $I_{\geom}$ consists of three terms: a central term, an sum of orbital integrals of elliptic elements, and a sum of constant terms of diagonal elements. By Lemma \ref{lem9.2}, there are only finitely many nonvanishing orbital integrals, and there are only finitely many constant terms because in a given compact subset of $T(\AAA^{\infty})$ there are only finitely many cosets $T(F)/Z(F)$.
The central term of $(-1)^{[F:\QQ]}I_{\geom,\chi}(Z(\AAA),\, \chi,\, \phi_S\cdot \phi_{\lambda} \cdot \phi_{\xi})$ is simply
$$(-1)^{[F:\QQ]}\tau_Z(G) \cdot\phi_S(1)\cdot \phi_{\lambda}(1)\cdot \phi_{\xi}(1) = 
	\tau_Z(G)\cdot\phi_S(1)\cdot \phi_{\lambda}(1)\cdot \dim(\xi)$$
so upon dividing by $\tau_Z(G)\cdot\phi_{\lambda}(1)\cdot \dim(\xi)$ we are left with $\phi_S(1) = \mupl_{S,\chi}(\phi_S)$.

Each orbital integral term in $I_{\geom,\chi}$ is of the form
$$ D(\gamma)\cdot O_{\gamma_S}(\phi_S)\cdot O_{\gamma^{S,\infty}}(\phi_\lambda) \cdot O_{\gamma_\infty}(\phi_{\xi}).$$
As we let $\lambda \to \infty$, the only nonconstant term is $O_{\gamma^{S,\infty}}(\phi_\lambda)$. Upon dividing by $\phi_\lambda(1)$ and taking absolute values, this is bounded by 
$$O_{\gamma^{S,\infty}}(\one_{\Gamma_0(\nn_\lambda)\cdot Z(\AAA^{S,\infty})}).$$
This goes to zero by Corollary \ref{cor8.9}. The same argument shows that each constant term vanishes asymptotically. 

As such, we have
\begin{align*}
	\lim_{\lambda \to \infty} \mudisc_\lambda(\phi_S) 
	& = \lim_{\lambda\to\infty} 
		(-1)^{[F:\QQ]}\frac
			{I_{\geom}(Z(\AAA),\, \chi,\, \phi_S\cdot \phi_{\lambda}\cdot \phi_{\xi})}
		 	{\tau_Z(G)\cdot \phi_{\lambda}(1)\cdot \dim(\xi)}
	\\ & = \phi_S(1)
	\\ & = \mupl_{S,\chi}(\phi_S)
\end{align*}

This completes the proof of the equidistribution theorem for Plancherel transforms $\widehat\phi_S$ of functions $\phi_S \in \HH(\GL_2(F_S),\, Z(F_S),\, \chi_S)$. When $\widehat f_S \in \mathscr{F}_0(\GL_2(F_S)^{\wedge})$ is arbitrary, we use Sauvageot's density theorem for fixed central character. (This is exactly as in Shin and Templier's proof of Corollary 9.22 in \cite{ST12}, except that we use Sauvageot's density theorem for fixed central character. We repeat the proof here for completeness).

Fix $\epsilon > 0$ and pick $\phi_{S},\, \psi_{S}\in \HH(\GL_2(F_S),\, Z(F_S),\, \chi)$ such that $|\widehat f_S - \widehat \phi_S| \leq \widehat \psi_S$ on $\GL_2(F_S)^{\wedge,\chi}$ and so that $\mupl_\chi(\widehat \phi_S) < \epsilon/3$. Then we have
\begin{align*}
	|\mupl_{\chi}(\widehat f_S) - \mudisc_\lambda(\widehat f_S)|
	& \leq |\mupl_{\chi}(\widehat f_S - \widehat \phi_S)| 
		+ |\mupl_\chi(\widehat \phi_S) - \mudisc_\lambda(\wh \phi_S)|
		+ |\mudisc_\lambda(\widehat \phi_S - \widehat f_S)|
	\\ & \leq |\mupl_{\chi}(\wh \psi_S)| 
		+ |\mupl_{\chi}(\widehat \phi_S) - \mudisc_\lambda(\wh\phi_S)|
		+ |\mudisc_{\lambda}(\widehat \psi_S)|
\end{align*}
The first term is at most $\epsilon/3$. The second term approaches $0$ as $\lambda \to \infty$, so it is eventually at most $\epsilon/3$. The third term approaches $|\mupl(\widehat \psi_S)| < \epsilon/3$ as $\lambda \to \infty$, so for large $\lambda$ it is eventually at most $\epsilon/3$. Therefore, for large $\lambda$ we have $|\mupl_{\chi}(\widehat f_S) - \mudisc_\lambda(\widehat f_S)| < \epsilon$, finishing the proof.
\end{proof}

\begin{cor} \label{cor9.3} Fix a weight $k$ and let $\chi$ of conductor $\ff$ occuring in weight $k$. Let $(\nn_\lambda)$ be a sequence of levels divisible by $\ff$. Then
$$
	\dim S_k(\Gamma_1(\nn_\lambda),\, \chi) 
	= \tau_Z(G)\cdot [\GL_2(\oo_F): \Gamma_0(\nn)]\cdot \dim(\xi_k) 
	+ o(N(\nn)^{1/2})
$$
as $\lambda \to \infty$.

If $F \neq \QQ$ then the error term is $o(N(\nn)^{\epsilon})$.
\end{cor}
\begin{proof} Apply the above to $S = \emptyset$ with $\phi_S = 1$, use the bounds on the constant terms and orbital integrals in Proposition \ref{prop8.8} and Corollary \ref{cor8.9}, and note that when $F \neq \QQ$, there are no constant terms on the geometric side of the trace formula.
\end{proof}

\begin{rem}\label{rem9.4} 
It is worth here comparing our results to those of \cite{Wei09}. First, Shin has computed $\tau_Z(\GL_2/F) = (-1)^{[F:\QQ]}\zeta_F(-1)2^{1 - [F:\QQ]}$ (see (iii) in the proof of Lemma 6.2 in \cite{Shi12}), whereas Weinstein's main term counting the number of cusp forms of fixed inertial type is 
$$(-1)^{[F:\QQ]} \cdot \zeta_F(-1)\cdot 2^{1-[F:\QQ]}\cdot h_F \cdot \dim(\xi) \cdot [\GL_2(\oo_F): \Gamma_0(\nn)].$$ 
The discrepancy occurs because he fixes an inertial type, which only determines the central character on
$$Z(F)\cdot \widehat \oo_F^\times \cdot  F_\infty^\times \subseteq Z(\AAA).$$
This subgroup has index $h_F$. As such, given an inertial type $\tau$, and $\pi$ of inertial type $\pi$, $\chi_{\pi}$ may be one of $h_F$ different characters.
\end{rem}

We also have a Plancherel equidistribution theorem for newforms. Since the proof is the same in spirit as the Plancherel equidistribution theorem, we give a sketch:

\begin{cor} \label{cor9.5} Let $\chi$ be a character with conductor $\ff$ and let $\widehat f_S \in \mathscr{F}_0(\GL_2(F_S)^{\wedge,\chi})$. Let $(\nn_\lambda) \to \infty$ be a sequence of levels coprime to $S$ with $\ff^S \mid \nn_\lambda$ and $N(\nn_{\lambda}) \to \infty$. Then
	$$\lim_{\lambda\to\infty} \mucusp_{\phi^{\new}_{\nn_\lambda,\chi} \, \xi,\, \chi}(\widehat f_S) 
	= \lim_{\lambda\to\infty} \mudisc_{\phi^{\new}_{\nn_\lambda,\chi} \, \xi,\, \chi}(\widehat f_S) 
	= \mupl_\chi(\widehat f_S).$$
\end{cor}
\begin{proof} We can assume the conductor $\ff$ is not divisible by any primes of norm 2. In this case, a quick computation shows
$$\frac
	{\phi^{\new}_{\nn_{\lambda},\chi,\pp}}
	{\phi^{\new}_{\nn_{\lambda},\chi,\pp}(1)}
= c_0 \frac
		{\phi_{\nn_{\lambda},\chi,\pp}}
		{\phi_{\nn_{\lambda},\chi,\pp}(1)}
	+ 2c_1N(\pp)^{-1} \frac
		{\phi_{\nn_{\lambda}/\pp,\chi,\pp} }
		{\phi_{\nn_{\lambda}/\pp,\chi,\pp}(1)}
	+ c_2N(\pp)^{-2} \frac
		{\phi_{\nn_{\lambda}/\pp^2,\chi,\pp}}
		{\phi_{\nn_{\lambda}/\pp^2,\chi,\pp}(1)}
$$
where $c_0,\, c_1,\, c_2$ are real constants of absolute value at most 2. 

Therefore, if $\ord_{\pp}(\nn) = r$, Proposition \ref{prop8.8} tell us that the orbital integrals
$$O_{\gamma_\pp}({\phi^{\new}_{\nn_{\lambda},\chi,\pp}})$$
are bounded in absolute value by $16 \cdot O_{\gamma}(\one_{Z(F_\pp)K_{\pp}})^2\cdot N(\pp^r)$. As such, we get a bound on the \emph{global} orbital integral:
$$|O_{\gamma}(\phi_S\phi^{\new}_{\nn,\chi}\phi_\infty)| \leq 
C' 16^{P(\nn)} O_{\gamma}(\one_{Z(\AAA^{S,\infty})K^{S,\infty}})^2N(\nn)^{-1},$$
for some constant $C'$ depending only on $\phi_{\infty}$ and $\phi_S$; this goes to zero as $o(N(\nn)^{-1 + \epsilon})$.

The analogous proof works for constant terms, and the corollary follows exactly as in Theorem \ref{thm9.1}; we apply the trace formula and then use Sauvageot density to adapt to the case when $\wh f_S \in \mathscr{F}_0(\GL_2(F_S)^{\wedge,\chi})$ is arbitrary.
\end{proof}


\section{Proof of the Main Theorem}\label{sec:10}
The goal of this section is to prove the our Main Theorem:

\begin{thm} \label{thm10.1}Let $F$ be a totally real field, $k$ a weight with $k_1 \equiv\ldots\equiv k_n$ mod $2$, $\chi$ an automorphic character occurring in weight $k$, and $(\nn_\lambda)$ a sequence of levels. Let $B_k(\Gamma_1(\nn),\,\chi)$ denote the standard basis of Hecke eigenforms of weight $k$, level $\nn$, and character $\chi$. For any $A \in \ZZ_{\geq 1}$, let $B_k(\Gamma_1(\nn),\, \chi)_{\leq A}$ be the subset consisting of those forms with $[\QQ(f):\QQ] \leq A$. Then
$$\lim_{\lambda \to \infty} \frac
	{\#B_k(\Gamma_1(\nn_\lambda),\,\chi)_{\leq A}}
	{\#B_k(\Gamma_1(\nn_\lambda),\,\chi)}
= 0.$$
\end{thm}

\begin{proof}
We'll begin the proof with three lemmas:

\begin{lem} \label{lem10.2} Fix $\epsilon > 0$. Then there is a $P_0\in \ZZ$ so that, for all rational primes $p > P_0$, all places $\pp| p$, and all levels $\nn$ with $\ord_\pp(\nn) \geq 3$, we have
$$\frac
	{\#B_k(\Gamma_1(\nn),\,\chi)_{\leq A}}
	{\#B_k(\Gamma_1(\nn),\,\chi)} < \epsilon.$$
\end{lem}
\begin{proof} By Proposition \ref{prop3.6}, for $p > 2A + 1$, if $\pp \mid p$, $\ord_\pp(\nn) \geq 3$, and $f$ is a \emph{newform} of level $\nn$, then $[\QQ(f):\QQ] > A$. We will henceforth assume $p > 2A + 1$.

Let $\pp \mid p$, let $f \in B_k(\Gamma_1(\nn),\, \chi)$ and assume $f$ satisfies $[\QQ(f): \QQ] \leq A$. As such, $f$ must come from a newform of level $\dd$ with $\ord_\pp(\dd) \leq 2$.

Write $\nn = \pp^B\ft$. By \cite{AL70} We can write
$$S_k(\Gamma_1(\nn),\,\chi) 
= \bigoplus_{\ff(\chi) \mid \dd \mid \nn} 
	S_k^{\new}(\Gamma_1(\dd),\,\chi)^{d(\nn/\dd)}$$
and let 
$$S_k(\Gamma_1(\nn),\,\chi)^{\leq 2} = 
\bigoplus_{\substack{\ff(\chi) \mid \dd \mid \nn \\ \ord_\pp(\dd) \leq 2}} 
	S^{\new}_{k}(\Gamma_1(\dd),\,\chi)^{d(\nn/\dd)}.$$

We claim
$$\dim S_k(\Gamma_1(\nn),\,\chi)^{\leq 2} \leq (B - 1) \dim S_k(\Gamma_1(\pp^2\ft),\,\chi).$$
To prove this, it's enough to show that for any $\dd$ dividing $\pp^2 \ft$, the multiplicity of $S_k^{\new}(\Gamma_1(\dd),\, \chi)$ in $S_k(\Gamma_1(\nn),\, \chi)$ is bounded above by $(B - 1)$ times its multiplicity in $S_k(\Gamma_1(\pp^2\ft),\, \chi)$. We note that the multiplicity of $S_k^{\new}(\Gamma_1(\pp^b \fa),\,\chi)$ in $S_k(\Gamma_1(\nn),\,\chi)$ is 
$$(B - b + 1)d(\ft/\fa),$$ and the multiplicity of $S_k^{\new}(\Gamma_1(\pp^b \fa),\,\chi)$ in $S_k(\Gamma_1(\pp^2\ft),\,\chi)$ is $(3 - b)d(\ft/\fa)$. Since $x + y + 1 \leq (x + 1)(y + 1)$ for nonnegative $x,\,y$, we have
$$(B - b + 1) \leq (B - 1)(3 - b)$$
proving the claim.

By Corollary \ref{cor9.3}, for large $N(\nn)$, we have 
$$\alpha [\GL_2(\oo_F): \Gamma_0(\nn)] \leq \dim S_k(\Gamma_1(\nn),\,\chi) \leq \beta [\GL_2(\oo_F): \Gamma_0(\nn)]$$
for some constants $\alpha,\, \beta$. In particular, we'll assume $p$ is large enough that this holds whenever $N(\nn) \geq p^2$.

Therefore, assuming $p$ is large enough, we have
\begin{align*} 
	\frac
		{\dim S_k(\Gamma_1(\nn),\,\chi)^{\leq 2}}
		{\dim S_k(\Gamma_1(\nn),\,\chi)}
	& \leq (B - 1) \frac
		{\dim S_k(\Gamma_1(\pp^2\ft),\,\chi)}
		{\dim S_k(\Gamma_1(\nn),\,\chi)}
	\\ & \leq (B - 1)\frac
		{\beta\cdot [\GL_2(\oo_F):\Gamma_0(\pp^2)]\cdot [\GL_2(\oo_F):\Gamma_0(\ft)]}
		{\alpha\cdot [\GL_2(\oo_F):\Gamma_0(\pp^B)]\cdot [\GL_2(\oo_F):\Gamma_0(\ft)]}
\\ & \leq \frac{\beta}{\alpha} (B - 1) N(\pp)^{2-B}.\end{align*}

Note that $(B - 1)N(\pp)^{2-B}$ is decreasing on $B \geq 3$ assuming $N(\pp) \geq 2$, so we just need to pick $p$ large enough so that $\frac{p - 1}{2} > A$ and $\frac{\beta}{\alpha}\frac{2}{p} < \epsilon$, finishing the proof.
\end{proof}

Recall from Section 3 that, given an orbit $\OO$, $\QQ(\OO)$ is the intersection of $\QQ(\pi)$ for $\pi \in \OO$.

\begin{lem} \label{lem10.3} Fix a character $\chi$ and $\epsilon > 0$. Let $B_{\pp} \subseteq \GL_2(F_{\pp})^{\wedge, t,\chi_{\pp}}$ be the union of those orbits $\OO_{\chi_\pp}$ with $c(\OO_{\chi_\pp}) \leq 2$ and $[\QQ(\OO_{\chi_{\pp}}):\QQ] > A$. Then if $\pp$ is coprime to $\ff(\chi)$ and $N(\pp)$ is sufficiently large, we have
$$\frac
	{\mupl_{\chi_{\pp}}(B_\pp)}
	{\mupl_{\chi_{\pp}} (\widehat {\phi}_{\pp^2,\chi,{\pp}})} > 1 - \epsilon$$
where $\phi_{\pp^2,\chi,{\pp}}$ is the $\pp$-component of the function $\phi_{\pp^2,\chi}$ defined in \ref{defn7.5}.
\end{lem}
\begin{proof} Given $A$, let $f(A)$ be large enough that if $\zeta$ is any root of unity with $[\QQ(\zeta):\QQ] \leq 2A$, then $\zeta \in \QQ(\zeta_{f(A)})$. Because $(\oo_{F}/\pp)^{\times}$ is cyclic of order $q - 1$, there are at least $q - 1 - f(A)$ characters $\chi_0: \oo_{F,\pp}^\times \to \CC^\times$ of conductor $1$ with $[\QQ(\chi_0(x)): \QQ] > 2A$ for some $x$. As such, there are at least $\frac{q - 1 - f(A)}{2}$ conjugate pairs of such characters. If $\OO$ is the principal-series orbit corresponding to such a pair, then $[\QQ(\OO):\QQ] > A$ by Proposition \ref{prop3.6}. By Computation \ref{comp6.13}, each such orbit has Plancherel measure $q + 1$, so the total Plancherel measure of these orbits is at least $\frac{(q + 1)(q - 1 - f(A))}{2}$.

We must also count supercuspidal orbits. We note that if $L'/L$ is ramified, then the map $(\oo_L/\varpi_L)^\times \to (\oo_{L'}/\omega_{\LL'})^\times$ is an isomorphism, so there are no characters $\eta_0: \oo_{L'}^{\times} \to \CC^\times$ of conductor $1$ such that $\eta_0 \neq \overline \eta_0$. Therefore, all such supercuspidal representations of conductor $2$ come from characters $\eta: L'^\times \to \CC^\times$ where $L'/L$ are unramified, and $\eta$ has conductor $1$. The central character of $\pi_{\eta}$ is the map $x \mapsto (-1)^{v_L(x)}\cdot \eta(x)$.

As such, if $\eta$ has conductor $1$ and $\pi_{\eta}$ has unramified central character, then $\eta_0$ factors through $\oo_{L'}^\times/(\oo_L^\times (1 + \pp_{L'}))$; this is cyclic of order $q + 1$. Let $\eta_0$ send the generator to a root of unity $\zeta$. Then $\zeta^{q+1} = 1$, and if $\eta = \overline \eta$ we have $\zeta^{q} = \zeta$; thus $\eta_0 \neq \bar \eta_0$ if $\eta_0$ does takes values outside of $\pm 1$. As such, there are $\frac{q + 1 - f(A)}{2}$ conjugate pairs $\{\eta_0,\, \overline \eta_0\}$ with $\eta_0(x) = \zeta_r$ for $r > f(A)$. 

Therefore, there are at least $\frac{q + 1 - f(A)}{2}$ supercuspidal representations $\pi$ of conductor $2$ and fixed unramified central character with $[\QQ(\pi):\QQ] > A$. Each has formal degree $q - 1$ by Computation \ref{comp6.13} (4), and so the total measure of these is $(q - 1)\frac{q + 1 - f(A)}{2}$. Therefore, the total Plancherel measure of all the orbits of conductor $2$ with large enough field of rationality is at least $q^2 - 1 - f(A)$.

On the other hand, 
$$\mupl_{\chi} (\widehat \phi_{\pp^2,\chi,{\pp}})
= {\phi}_{\pp^2,\chi,{\pp}}(1) = q(q+ 1)$$
and 
$$\lim_{q\to\infty} \frac
	{q^2 - 1 - f(A)}{q(q+1)} = 1$$
	completing the proof.
\end{proof}

\begin{lem} \label{lem10.4} Fix a character $\chi$ and let $\pp$ be any prime that does not divide the conductor $\ff$ of $\chi$. Let $(\nn_{\lambda})$ be a sequence of levels that are divisible by $\ff$ and prime to $\pp$. Then 
$$\lim_{\lambda \to \infty} \frac
	{\# B_k(\Gamma_1(\nn_\lambda),\,\chi)_{\leq A}}
	{\# B_k(\Gamma_1(\nn_\lambda),\,\chi)}
= 0.$$
\end{lem}
\begin{proof} This argument is due to Serre in the case $F = \QQ$ (see \cite{Ser97}) and all changes are simply cosmetic. We repeat it here for reference and completeness.

Let $f$ be a modular form of level $\nn_{\lambda}$ with $[\QQ(f):\QQ] \leq A$. Then $[\QQ(a_{\pp}(f)): \QQ] \leq A$. Because $\nn_{\lambda}$ is prime to $\pp$, then $a_{\pp}$ is the sum of two Weil-$q$-integers $\alpha,\, \beta$ of weight $\max\{k_1\} - 1$ \cite[Theorem 1.4 (1)]{RT11}. Since $[\QQ(a_\pp):\QQ] \leq A$, then $Q(\alpha),\, Q(\beta)$ are of degree at most $2A$ over $\QQ$. Thus $\alpha,\, \beta$ must satisfy monic polynomials in $\ZZ[x]$ of degree at most $2A$ and whose coefficients are bounded for fixed $q$ and $\max \{k_i\}$; thus $\alpha,\,\beta$ take only finitely many values. 

In this case, $\pi_{f,\pp}$ is the unramified representation $\chi_1 \times \chi_2$, where $\chi_1,\, \chi_2$ take a uniformizer to $\alpha,\, \beta$. As such, $\pi_{f,\pp}$ takes only finitely many possible values in the unramified orbit $\OO^{\chi}$. Let $\wh \one_{0, \leq A}$ be the characteristic function of this finite set of points. Because the Plancherel measure on the unramified orbit has no points of positive measure, we have $\mupl_{\chi_\pp} (\one_{0,\,\leq A}) = 0$, and because it is supported on fintiely many points, it lives in $\mathscr{F}_0(\GL_2(F_\pp)^{\wedge})$. As such, the Plancherel equidistribution theorem completes the lemma.
\end{proof}


With this in hand, we complete the proof of the Main Theorem.

Fix $A \in \ZZ_{\geq 0}$ and $\epsilon > 0$. Fix $P > 2A + 1$ so that for any rational prime $p > P$, and any prime $\pp$ of $F$ lying above $p$, we have
$$\frac
	{\mupl_{\chi}(B_\pp)}
	{\mupl_{\chi} (\widehat {\phi}_{\pp^2,\chi,\pp})} > 1 - \epsilon$$
(where $B_\pp$ is as in Lemma \ref{lem10.3}), and such that for any $\nn$ with $\ord_{\pp}(\nn) \geq 3$, we have
$$\frac
	{\#B_k(\Gamma_1(\nn),\,\chi)_{\leq A}}
	{\#B_k(\Gamma_1(\nn),\,\chi)} < \epsilon$$

Fix, once and for all, primes $\pp_1,\ldots,\, \pp_r$ coprime to the conductor $\ff$, with $\pp_i \mid p_i > P$, and so that
$$\prod_{i = 1}^r \frac
	{q_i - 1}
	{q_i + 1}
< \epsilon.$$
This is possible because the Dedekind zeta function $\zeta_F$ has a pole at $1$, and because $\frac{q_i - 1}{q_i + 1} \leq 1 - \frac{1}{q_i}$.

With this in hand, let $t = (t_1,\,\ldots,\, t_r)$ be a tuple, where $t_i$ is either $0, \, 1,\, 2$, or `$\geq 3$'. For such a tuple, define the subsequence $(\nn_{\lambda, t})$, where $\nn_{\lambda} \in \{\nn_{\lambda, t}\}$ if
$\ord_{\pp_i}(\nn_{\lambda}) = t_i$ for $i = 1,\ldots,\, r$ (and where if $t_i =$ `$\geq 3$' then $\ord_{\pp}(\nn_{\lambda,t}) \geq 3$). This breaks the sequence $(\nn_\lambda)$ into finitely many subsequences.  We can assume each is either empty or infinite.

We'll show that for every $t$, if $\lambda$ is sufficiently high we have
$$\frac
	{\# B_k(\Gamma_1(\nn_{\lambda,t}),\,\chi)_{\leq A}}
	{\# B_k(\Gamma_1(\nn_{\lambda,t}),\,\chi)}
< \epsilon.$$

If $t_i$ is `$\geq 3$' for some $i$, we are done because each $\pp_i$ is chosen to be sufficiently large. If $t_i = 0$ for some $i$, we are done by Lemma \ref{lem10.4}.

If $t_i = 2$ for some $i$, let $S = \{\pp_i\}$, and let $\nn_{\lambda,t}' = \nn_{\lambda,t}/\pp_i^2$. Then we have 
\begin{align*}
	\frac
		{\# B_k(\Gamma_1(\nn_{\lambda,t}),\,\chi)_{\leq A}}
		{\# B_k(\Gamma_1(\nn_{\lambda,t}),\,\chi)}
	& \leq 
		1 - \frac
			{|\FF_{\cusp,\chi}(\widehat \one_{B_{\pp_i}},\, \widehat \phi_{\nn_{\lambda,t}', \chi},\, \xi_k)|}
			{|\FF_{\cusp,\chi}(\widehat \phi_{\pp_i^2,\, \chi},\,\widehat\phi_{\nn_{\lambda,t}', \chi},\, \xi_k)|}
	\\ & \to 1 - \frac
			{\mupl_{\pp_i,\,\chi}(B_{\pp_i})}
			{\mupl_{\pp_i,\,\chi}(\widehat \phi_{\pp_i^2,\,\chi})}
	\\ & < \epsilon
\end{align*}
(where the third line follows by the Plancherel equidistribution theorem)
so that eventually we have 
$$	\frac
		{\# B_k(\Gamma_1(\nn_{\lambda,t}),\,\chi)_{\leq A}}
		{\# B_k(\Gamma_1(\nn_{\lambda,t}),\,\chi)}
	< \epsilon.$$
	
The only remaining case is $t_1 = \ldots = t_r = 1$. Let $C_{\pp}$ be the finite set consisting of the two Steinberg points of conductor $1$ and central character $\chi_{\pp}$ along with the finite set of points $\{\pi_\pp\}$ in the unramified orbit where the Frobenius eigenvalues are Weil-$q$-integers of small enough degree. As in Lemma \ref{lem10.4}, the characteristic function $\widehat\one_{C_\pp}$ lives in $\mathscr{F}_0(\GL_2(F_\pp)^{\wedge})$ (the set of discontinuities is a finite set in the unramified orbit), and moreover we have
$$\mupl_{\chi,\,\pp}(\widehat \one_{C_\pp}) = {q - 1},$$
because there are two Steinberg representations with central character $\chi_{\pp}$ (corresponding to its two square roots) and because each has formal degree $\frac{q-1}{2}$. 

As such, if we take $S = \{\pp_1,\ldots,\, \pp_r\}$ and
$$\widehat f_S = \prod_{i = 1}^r \wh\one_{C_{\pp}},$$
then the same logic as in the $t_i = 2$ case tells us:
\begin{align*} 
\limsup_{\lambda \to \infty} \frac
	{\# B_k(\Gamma_1(\nn_{\lambda,t}),\,\chi)_{\leq A}}
	{\# B_k(\Gamma_1(\nn_{\lambda,t}),\,\chi)}
& \leq \frac
	{\mupl_{S,\chi}(\widehat f_S)}
	{\mupl_{S,\chi}(\widehat \phi_{\pp_1\ldots \pp_r,\chi})}
\\ & = \prod_{i = 1}^r \frac
	{\mupl_{S,\chi}(C_\pp)}
	{\phi_{\pp_i,\chi}(1)}
\\ & = \prod_{i = 1}^r \frac{q_i - 1}{q_i + 1}
\\ & < \epsilon\end{align*}
so in particular, we eventually have 
$$ \frac
	{\# B_k(\Gamma_1(\nn_{\lambda,t}),\,\chi)_{\leq A}}
	{\# B_k(\Gamma_1(\nn_{\lambda,t}),\,\chi)} 
< \epsilon$$
for this subsequence. This completes the proof.
\end{proof}

We also have a partial result for fields of rationality of newforms, which we briefly discuss. We would like to say that, as the norm of our ideals increases, that a smaller percentage of the set of newforms has bounded field of rationality. But there is an obstruction to our methods. Consider, for example, the case where $F = \QQ$, $\chi$ is the trivial character, and $(\nn_\lambda)$ is the sequence
$$(2, \, 2\cdot 3,\, 2\cdot 3 \cdot 5, \ldots).$$
In this case, for any prime $p$, our sequence will eventually satisfy $\ord_p(\nn_\lambda) = 1$, so in particular, the associated representation will always be Steinberg. Indeed, if we pick any finite set of primes $S$ and look at the representations $\pi_{f,S}$, the field of rationality at these places will eventually be $\QQ$. However, the corollary below shows that this is, in effect, the only obstruction.

\begin{cor} \label{cor10.5} Let $B_k^{\new}(\Gamma_1(\nn),\chi)$ be the canonical basis of newforms of weight $k$, character $\chi$ occuring in weight $k$, and level $\nn$, and define $B_k^{\new}(\Gamma_1(\nn),\chi)_{\leq A}$ as above. Let $(\nn_{\lambda})$ be a sequence of levels satisfying the following condition:

$$\text{For any $P \in \ZZ$, there is a finite set $\pp_1,\ldots,\, \pp_r$ with $\pp_i \mid p_i > P$,}$$
$$\text{and for all sufficiently high $\lambda$, there is an $i$ so that $\ord_{\pp_i}(\nn_\lambda) \neq 1$.}$$

Then
$$\lim_{\lambda \to \infty} \frac
	{B_k^{\new}(\Gamma_1(\nn_{\lambda}),\chi)_{\leq A}}
	{B_k^{\new}(\Gamma_1(\nn_{\lambda}),\chi)}
= 0.$$
\end{cor}

\begin{proof} Fix $\epsilon > 0$ and let $P$ be large enough that the results of Lemmas \ref{lem10.2} and \ref{lem10.3} hold. By hypothesis we can find primes $\pp_1,\ldots,\, \pp_r$ such that, for $\lambda$ sufficiently high we have some $\ord_{\pp_i}(\nn_\lambda) \neq 1$ for some $i$. Break the sequence into subsequences $(\nn_{\lambda, i, x})$ where $i = 1,\ldots,\, r$ and $x = 0, 2$ or `$\geq 3$'; say $\nn_{\lambda} \in \{\nn_{\lambda, i, x}\}$ if $\ord_{\pp_i}(\nn_\lambda) = x$. Each subsequence eventually satisfies:
 $$\frac
	{B_k^{\new}(\Gamma_1(\nn_{\lambda,i,x}),\chi)_{\leq A}}
	{B_k^{\new}(\Gamma_1(\nn_{\lambda,i,x}),\chi)}
< \epsilon$$
by the same logic as in the proof of the main theorem.
\end{proof}

\section{Appendix: Proofs of Fixed-Central-Character Representation-Theoretic Prerequesites}
\label{sec:11}
In this section we will prove some properties of the fixed central character Plancherel measure for $\GL_2$.  We begin with a couple of remarks: first, the proofs rely on little more than abelian Fourier analysis after we accept the non-fixed central character versions of these results.  Second, we have written the proofs to be as applicable as possible to larger groups.  In fact, we hope to use analogous results in future works.  We begin with some very basic lemmas on Hecke algebras.

\subsection{Basic Results on Hecke Algebras}
\label{subsec:11.1}

Recall the definition of the a fixed central datum Hecke algebra.  As before, $R$ is either a local field $L$ or the ring of ad\'{e}les over a number field $F$:

\begin{defn}\label{defn11.1} Let $\XX$ be a closed subgroup of the center $Z(\GL_2(R))$, and let $\chi: \XX \to \CC^{\times}$ be a unitary character. The \emph{Hecke algebra} $\HH(\GL_2(R),\, \XX,\,\chi)$ is the convolution algebra of smooth functions $\phi: \GL_2(R) \to \CC$ that are compactly-supported modulo $\XX$ and that satisfy the transformation property 
$$\phi(gx) = \phi(g)\chi(x)^{-1}\text{ for all $g\in G,\, x\in \XX$.}$$
\end{defn}

\begin{defn}\label{defn11.2}
Given $(\XX,\, \chi),\, (\XX',\,\chi')$ with $\XX \supseteq \XX'$ and $\chi|_{\XX'} = \chi'$, we define the \emph{averaging map} 
\begin{align*} 
	\HH(\GL_2(R),\, \XX',\,\chi') & \to \HH(\GL_2(R),\, \XX,\,\chi)
 	\\ \phi & \mapsto \overline \phi_{\XX, \chi}
\end{align*}
where $\overline \phi_{\XX, \chi}$ is defined by
$$\overline \phi_{\XX, \chi}(g) = \int_{\XX'\bs \XX} \phi(gx^{-1}) \chi(x)^{-1}\, dx.$$

In this case, we say $\overline \phi_{\XX,\,\chi}$ is the \emph{average} of $\phi$ over $\XX$ with respect to $\chi$.
\end{defn}

The rest of this subsection will be a sequence of Fourer-analytic lemmas that will be useful in the proofs of the fixed-central-character trace formula, Plancherel theorem, and Sauvageot theorem.

\begin{lem}\label{lem11.3} Let $(\XX,\,\chi),\, (\XX',\,\chi')$ be as above and let $\phi \in \HH(\GL_2(R),\, \XX',\, \chi')$. Assume the Haar measures on $\GL_2(R)/\XX,\, \GL_2(R)/\XX',$ and $\XX/\XX'$ are chosen compatibly. Then for any $\pi$ with $\chi_{\pi}|_{\XX} = \chi$ we have
$$\tr_{\XX'} \pi(\phi) = \tr_{\XX} \pi(\overline \phi_{\XX,\chi}).$$
\end{lem}
\begin{proof} The proof is a quick application of Fubini's theorem:
\begin{align*}
	\tr_{\XX'} \pi(\overline \phi_{\XX',\chi'})
	& = \tr\left(\int_{\XX\bs \GL_2(R)} \int_{\XX' \bs \XX} \phi(gx)\, \chi(x)\pi(g)\, dx\,dg\right)
	\\ & = \tr\left(\int_{\XX'\bs \XX} \chi(x) \int_{\XX\bs \GL_2(R)} \phi(gx) \pi(g)\,dg\,dx\right)
	\\ & = \tr\left(\int_{\XX'\bs \XX} \chi(x) \int_{\XX\bs \GL_2(R)} \phi(g) \pi(gx^{-1})\,dg\,dx\right)
	\\ & = \tr\left(\int_{\XX'\bs \XX} \chi(x)\chi(x^{-1}) \int_{\XX\bs \GL_2(R)} \phi(g) \pi(g)\,dg\,dx\right)
	\\ & = \tr\left(\int_{\XX' \bs\GL_2(R)} \phi(g)\pi(g)\,dg\right)
	\\ & = \tr_{\XX} \pi(\phi)
\end{align*}
\end{proof}

\begin{defn}\label{defn11.4} Let $\phi \in \HH(\GL_2(R),\, \XX',\, \chi')$, let $\XX$ be a closed subset of $Z(\GL_2(R))$, and let $\Phi_{\XX}$ be smooth and compactly supported on $\XX$. We define the \emph{convolution} $\phi\star \Phi_{\XX}$ by
$$(\phi \star \Phi_{\XX})(g) = \int_{\XX} \phi(gx^{-1})\Phi_{\XX}(x)\, dx.$$
\end{defn}

\begin{defn}\label{defn11.5} Let $\Phi$ be smooth and compactly supported on $\XX$, a closed subset of $Z(\GL_2(R))$. We define its \emph{Fourier transform} $\widehat \Phi: \widehat \XX \to \CC$ by
$$\widehat \Phi(\chi) = \int_{\XX} \Phi(x)\chi(x)\,dx.$$
\end{defn}

Here, we apologize for the re-use of notation. We'll distinguish the following cases: when we use lower-case Greek letters like $\phi,\, \psi$, we mean elements of some Hecke algebra on $\GL_2(R)$, and $\widehat \phi,\, \widehat \psi$ will denote their \emph{Plancherel} transforms. Upper-case Greek letters $\Phi,\, \Psi$ will always denote functions on some closed subset of the center, and $\widehat \Psi,\, \widehat \Phi$ will always denote their \emph{Fourier} transforms as functions on a locally compact abelian group.

\begin{lem}\label{lem11.6} Let $\Phi_{\XX}$ be as above and let $\phi \in \HH(\GL_2(R),\, \XX',\, \chi')$ with $\XX\supseteq \XX'$. Then
\begin{enumerate} 
	\item $\phi \star \Phi_{\XX} \in \HH(\GL_2(R),\, \XX',\, \chi')$.
	\item For any $\pi$ with $\chi_\pi|_{\XX'} = \chi'$,
	$$\tr_{\XX'} \pi(\phi \star \Phi_{\XX}) = \widehat \Phi_{\XX}(\chi_{\pi}|_{\XX})\cdot \tr_{\XX'} \pi(\phi).$$
\end{enumerate} 
\end{lem}
\begin{proof} (1) is clear, so we will prove (2).
We have
\begin{align*}
	\tr_{\XX'} \pi(\phi \star \Phi_{\XX})
	& = \tr\left(\int_{\XX' \bs \GL_2(R)} \int_{\XX} \phi(gx^{-1})\Phi(x)\pi(g)\,dx\,dg\right)
	\\ & = \tr\left(\int_{\XX} \Phi(x) \int_{\XX'\bs\GL_2(R)} \phi(gx^{-1}) \pi(g)\,dg\,dx\right)
	\\ & = \tr\left(\int_{\XX} \Phi(x) \int_{\XX'\bs\GL_2(R)} \phi(g) \pi(gx)\,dg\,dx\right)
	\\ & = \tr \left(\int_{\XX} \Phi(x)\chi_{\pi}(x) \int_{\XX' \bs \GL_2(R)} \phi(g)\pi(g)\,dg\,dx\right)
\end{align*}
completing the proof.	
\end{proof}

\begin{lem}\label{lem11.7} Let $L$ be a $p$-adic field and let $\chi: L^{\times} \to \CC^{\times}$ be a character.
Define
$$\Phi_{M,\chi}(z) = \frac{1}{M^2} \sum_{j = 0}^{M-1} \int_{-j \leq v_L(z) \leq j} \chi^{-1}(z)\,dz.$$
Then \begin{enumerate}
	\item if $\chi$ and $\tau$ differ by a ramified character, then $\widehat \Phi_{M,\,\chi}(\tau) = 0$,
	\item if $\tau \neq \chi$ then $\widehat \Phi_{M,\,\chi}(\tau) \to 0$ as $M \to \infty$, and 	
	\item if $\tau = \chi$ then $\widehat \Phi_{M,\, \chi}(\tau) = 1$ for all $M$.
\end{enumerate}
\end{lem}

\begin{proof} We have 
\begin{align*} 
	\widehat \Phi_{M,\,\chi}(\tau) 
	& = \frac{1}{M^2} \sum_{j = 0}^{M-1} \int_{q_L^{-j} \leq |z| \leq q_L^j} (\tau\chi^{-1})(z)\,dz
	\\ & = \frac{1}{M^2} \left(\sum_{i = 1 - M}^{M - 1} 
		(M - |i|)\cdot (\tau\chi^{-1})(\varpi)^i\right)
		\left(\int_{\oo_L^\times} (\tau \chi^{-1})(z)\,dz\right)
\end{align*}
If $\tau\chi^{-1}$ is ramified, then the integral vanishes. If $\tau\chi^{-1} = 1$, then the integral is $1$ and the sum is $M^2$, completing (3). For (2), if $\tau \chi^{-1}$ is unramified, the sum is $\frac{1}{M}F_m(\tau\chi^{-1}(\varpi))$, where $F_M$ is the Fej\'er kernel from real Fourier analysis. For all $z \neq 1$ we have $\frac{1}{M} F_M(z) \to 0$, finishing the proof.
\end{proof}

\subsection{Fixed-Central-Character Plancherel Measure}
\label{subsec:11.3}

The goal of this section is to prove Proposition \ref{prop6.12}.  To begin, we construct a family of fixed-central-character Plancherel measures.  To this end, let $L$ be a local field and let $\chi:L^\times \to \CC^\times$ be a character, and let $\OO$ be an orbit in $\GL_2(L)^{\wedge,t}$.  Let $\OO_{\chi} = \OO \cap \GL_2(L)^{\wedge,t,\chi}$; then $\OO_{\chi}$ is either empty or an orbiforld of codimension 1 in $\OO$.  Henceforth, we'll assume $\OO_{\chi}$ is nonempty.

Let $\pi_0  = I_P^G(\omega_0)$ where $P$ is a parabolic subgroup with Levi subgroup $M$ and let $\omega$ be a discrete series representation of $M$.  Assume $\chi_{\pi_0} = \chi$.  Then are surjective maps 
$$X_u(M) \to \OO_M \to \OO$$
given by 
$$\tau \mapsto \omega \otimes \tau \mapsto I_P^G(\omega \otimes \tau).$$

Let $X_u(M)_0$ be the kernel of the restriction $X_u(M) \to X_u(Z(G))$. Then the above map restricts to a surjection $X_u(M) \to \OO_{\chi}$.  

Henceforth, if $Y \to Z$ is a finite surjective map of orbifolds, say $E \subseteq Y$ is a \emph{nice fundamental domain} if $E \to Z$ is bijective, and if there is an open set $U \subseteq E \subseteq \overline U$.

\begin{lem}\label{lem11.14} Let $E$ be a nice fundamental domain of $X_u(M)_0 \to \OO$, let $D_0$ be a nice fundamental domain of $X_u(G) \to X_u(Z(G))$, and let $D$ be the image of $D_0$ under $X_u(G) \to X_u(M)$. Then $D\times E \to X_u(M)$ is injective, and its image is a nice fundamental domain for $X_u(M) \to \OO$.
\end{lem}

\begin{proof} The proof is elementary after realizing that if $\tau \in X_u(M)$ and $d\in D$, if $\tau \mapsto \pi$ then $d\tau \mapsto \pi \otimes \tau$.
\end{proof}

Recall that the canonical measure on $\OO$ is chosen so that the maps $X_u(Z(M)) \leftarrow X_u(M) \rightarrow \OO$ locally preserve measures when the Haar measure on $X_u(M)$ has total measure $1$, and that the Plancherel measure is absolutely continuous with respect to the canonical measure.  If $\OO_{\chi}$ is nonempty, we define the \emph{canonical measure} on $\OO_{\chi}$ so that $X_u(M)_0 \to \OO_{\chi}$ locally preserves measures for a choice of a Haar measure on $X_u(M)_{0}$, and such that $\OO$ and $\OO_{\chi}$ have the same canonical measure.

\begin{prop}\label{prop11.15} Let $d\chi$ be the Haar measure on $\widehat{L^\times}$ giving each unramified orbit total measure $1$. Then for any $\wh f\in L^1(\GL_2(L)^{\wedge,t})$, we have
$$\int_{\GL_2(L)^{\wedge,t}} \wh f \, d\pi = \int_{\wh{L^\times}} \int_{\GL_2(L)^{\wedge,t,\chi}} \wh f\,d\pi_{\chi}\,d\chi.$$
\end{prop}
\begin{proof} By summing over all orbits, we just need to prove this for $\wh f$ supported on a single orbit $\OO$. Consider the nice fundamental domain $DE \subseteq X_u(M)$ constructed in Lemma \ref{lem11.14}; if we fix $d$, we that $dE$ maps surjectively and injectively onto $\OO_{\chi \otimes d_{Z(G)}}$, where $d_{Z(G)}$ is the restriction of $d$ to $Z(G)$. Therefore, if we pull back $\wh f$ to $\wh f_{DE}$ on $DE$, we have
\begin{align*} 
	\int_{\OO} \wh f\,d\pi
	& = \int_{DE} \wh f_{DE}(de)\, dd\,de
	\\ & = \int_D \int_E \wh f_{DE}(de)\, de\,dd
	\\ & = \int_{\OO_{\chi'}\neq \emptyset} \int_{\OO_{\chi'}} \wh f(\pi) \, d\pi_{\chi}\,d\chi
\end{align*}
where the last equality holds because fixing $dE$ maps surjectively and injectively onto $\OO_{\chi\otimes d_{Z(G)}}$, and because the total measure on $DE$ is the same as the total measure on $\OO$.
\end{proof}

\begin{cor}\label{cor11.16} Let $d\chi$ be as above and let $\nu^{\pl}$ be the Plancherel density function from the previous section, and let $\mupl_{\chi} = \nu^{\pl} d\chi$ on $\GL_2(L)^{\wedge, t,\chi}$.  For any $\wh f\in \mathscr{F}_0(\GL_2(L)^{\wedge})$, we have
$$\mupl(\wh f) = \int_{\widehat{L^\times}} \mupl_{\chi}(\wh f)\,d\chi.$$
\end{cor}
\begin{proof} Apply the above proposition to the function $\pi \mapsto \nu^{\pl}(\pi)\wh f(\pi)$.
\end{proof}

The measure above is the fixed-central-character Plancherel measure.   We must prove it has the properties stated in Proposition \ref{prop6.12}.  We'll prove the statements separately.

\begin{prop}\label{prop11.17} Fix $\phi\in \HH(\GL_2(L),\, Z(L),\, \chi)$.  We have
$$\phi(1) = \int_{\GL_2(L)^{\wedge, t, \chi}} \wh \phi(\pi) \,d\mupl_{\chi}(\pi).$$
\end{prop}
\begin{proof}
We will use the Fourier-theoretic ideas from Subsection \ref{subsec:11.1}. Let 
$$\phi_0(g) = \begin{cases} 
	\phi(g) & |\det(g)|_{L} = 1 \text{ or } q_L
	\\ 0 & \text{ otherwise}
\end{cases}$$
so that $\overline \phi_{0,\chi} = \phi$, and therefore if $\chi_{\pi} = \chi$ we have $\tr_{Z(L)} \pi(\phi) = \tr \pi(\phi_0)$ by Lemma \ref{lem11.3}.

Let $A_j$ be the annulus $\{z\in L^\times: -j \leq v_L(z) \leq j\}$, and let
$$\Psi_{M,\chi}(z) = M\Phi_{M,\chi} = \frac{1}{M}\sum_{j = 0}^{M-1} \one_{A_j}(z)\chi^{-1}(z)$$
where $\Phi_{M,\,\chi}$ is as in Lemma \ref{lem11.7}. Define
\begin{align*} 
	\phi_M(g) 
	& = (\phi_0 \star \Psi_{M,\chi})(g) 
	\\ & = \frac{1}{M} \sum_{j = 0}^{M-1} \int_{A_j} \phi_0(gz)\chi(z)\, dz
\end{align*}
so that 
$$\phi(1) = \lim_{M\to\infty} \phi_M(1) = \lim_{M\to\infty} \mupl(\widehat \phi_M).$$

Moreover, we have
\begin{align*} 
	\lim_{M \to \infty} \mupl(\widehat \phi_M)
	& = \lim_{M\to\infty}\int_{\pi} \tr \pi(\phi_M) \,d\mupl_\chi(\pi)
	\\ & = \lim_{M\to\infty}\int_{\chi' \in \widehat{L^{\times}}} \int_{\chi_\pi = \chi'} \tr \pi(\phi_M)\, d\mupl_{\chi'}(\pi) \,d\chi'
	\\ & = \lim_{M\to\infty}\int_{\chi' \in \widehat{L^{\times}}} \int_{\chi_\pi = \chi'} \widehat \Psi_{M,\chi}(\chi') \cdot \tr \pi(\phi_0) \, d\mupl_{\chi'}(\pi)\, d\chi' & \text{by Lemma \ref{lem11.6}}
	\\ & = \lim_{M\to\infty}\int_{\chi' \in \widehat{L^{\times}}} \widehat \Psi_{M,\chi}(\chi') \left(\int_{\chi_\pi = \chi'} \tr \pi(\phi_0) \, d\mupl_{\chi'}(\pi) \right) \,d\chi'
\end{align*}
By Lemma \ref{lem11.7}, $\widehat \Psi_{M,\chi}(\chi') = 0$ if $\chi'\chi^{-1}$ is ramified. If $\chi'\chi^{-1}$ is ramified, then $\widehat \Psi_{M,\chi}(\chi') = F_M(\chi'\chi^{-1}(\varpi))$, where $F_M: S^1 \to \CC$ is the Fej\'er kernel. Since 
$$\int_{S^1}F_M(s) \,ds = 1$$
and 
$$\lim_{M\to\infty} \int_{S^1 - U} F_M(s)\,ds = 0$$
for any open neighborhood of $1$ in $S^1$, then for any continuous function $h: S^1 \to \CC$ we have
$$\lim_{M\to\infty} \int_{S^1} F_M(s) h(s)\,ds = h(1).$$ As such, we have
\begin{align*}
	\lim_{M\to\infty}\int_{\chi' \in \widehat{L^{\times}}} \widehat \Psi_{M,\chi}(\chi') 
		\left(\int_{\chi_\pi = \chi'} \tr \pi(\phi_0) \, d\mupl_{\chi'}(\pi) \right) \,d\chi' 
	& = \int_{\chi_\pi = \chi} \tr \pi(\phi_0) \, d\mupl_{\chi}(\pi)
	\\ & = \int_{\chi_\pi = \chi} \tr_{Z(L)} \pi(\phi)\,d\mupl_{\chi}(\pi) & \text{by Lemma \ref{lem11.3}}
	\\ & = \mupl_{\chi}(\wh\phi)
\end{align*}
completing the proof.
\end{proof}

\begin{lem}\label{lem11.18} Let $d\pi$ be the canonical measure on $\OO$ and let $d\pi_{\chi}$ be the canonical measure on $\OO_{\chi}$.  If $d\mupl = \nupl d\pi$, then $d\mupl_{\chi} = \nupl d\pi_{\chi}$.
\end{lem}
\begin{proof} This is evident from the construction.\end{proof}

\begin{lem}\label{lem11.19} If $\pi$ is not a discrete series representation, then $\mupl_{\chi} = 0$.  If $\pi$ is a discrete series representation, then $\mupl_{\chi}(\pi) = d(\pi)$, the formal degree of $\pi$.
\end{lem}
\begin{proof} Assume $\pi\in \OO_{\chi}$ is not a discrete series, and that it is of the form $I_P^G(\omega)$ where $\omega$ is a discrete series representation of a Levi subgroup $M \neq G$.  Then $\dim X_u(M) = \dim(Z(M)) \geq \dim Z(G)$ and therefore $X_u(M)_0$ has positive dimension.  As such, the canonical measure on $\OO_{\chi}$ has no points of positive measure, and so neither does the Plancherel measure.

Now assume that $\pi\in \OO$ is a discrete series representation with central character $\chi$.  We will first show that the Plancherel density function $\nupl$ on $\OO$ is simply $\pi \mapsto d(\pi)$. To this end, we recall that the Plancherel density function is given by
$$\nupl(\pi) = c(G|M)^{-2} \gamma(G|M)^{-1} j(\omega) d(\omega)$$
where $\pi = I_G^P(\omega)$ and $M \leq P$ is a Levi subgroup.  Since $\pi$ is a discrete series representation, then $M = G$ and $\omega = \pi$.  From \cite[p. 240-241]{Wal03} we have that $c(G|G) = \gamma(G|G) = 1$.  Moreover, from the definition of $j$ as a composition of two intertwining operators on \cite[p. 236]{Wal03} it is clear that $j(\pi) = 1$.  This proves the claim.  In fact, if $\pi' = \tau \otimes \pi$ for \emph{any} character $\tau$ then $d(\pi) = d(\pi')$ by the definition of $d(\pi)$ on \cite[p. 265]{Wal03} (so that $\nupl$ is constant on a discrete series orbit).

We now compute the canonical measure of the orbit $\OO$.  Assume $X_u(G) \to X_u(Z(G))$ is an $n$-to-one cover and that $X_u(G) \to \OO$ is an $r$-to-one cover.  Then the canonical measure of $\OO$ is $n/r$.  

Therefore, it's enough to prove that $\OO_{\chi}$ consists of precisely $n/r$ isomorphism classes of discrete series representations.  First, because $X_u(G) \to X_u(Z(G))$ is an $n$-to-one cover then $X_u(G)_0$ has cardinality $n$; therefore, there is a surjection from a group of order $n$ to $\OO_{\chi}$ via a choice of basepoint. Finally, since $X_u(G) \to \OO$ is $r$-to-one, then so is $X_u(G)_0 \to \OO_{\chi}$ (note that if $\pi' = \tau \otimes \pi$, then $\pi'$ and $\pi'$ have the same stabilizers under the action of $X_u(G)$ by tensoring).  As such, each singleton in $\OO_{\chi}$ has canonical measure $1$, and therefore its fixed-central-character Plancherel measure is its formal degree.
\end{proof}

\begin{lem}\label{lem11.20} Sauvageot's density theorem holds for the fixed-central-character Plancherel measure.\end{lem}

\begin{proof} Let $D$ be a nice fundamental domain for $X_u(G) \to X_u(Z(G))$, and define $\wh f$ on $\GL_2(L)^{\wedge}$ by $\wh f(\pi) = \wh f_{\chi}(\pi \otimes d)$, where $d \in D$ is chosen uniquely so that $\chi_{\pi} \otimes d_{Z(G)} = \chi$. Note that $\wh f$ remains bounded and supported on a finite number of orbits. If $C_{\chi}$ is the set where $\wh f_\chi$ is continuous, then $\wh f$ is continuous on $C_{\chi}D$, which has measure equal to the measure of $\OO$.

Therefore, we can choose functions $\phi,\, \psi \in C_c^\infty(\GL_2(L))$ with $|\wh \phi - \wh f| \leq \psi$ and $\mupl(\wh \psi) \leq \epsilon$. Let $U_{\chi}$ be the unramified orbit of $\chi$ in $\wh{L^\times}$. Then we have
$$\int_{U_\chi} \mupl_{\chi'}(\wh \psi)\,d\chi' 
\leq \int_{\wh L^\times} \mupl_{\chi}(\wh \psi)\,d\chi'
< \epsilon$$
and since $U_{\chi}$ has measure $1$, there is a $\chi' \in U_{\chi}$ with $\mupl_{\chi'}(\wh\psi) \leq \psi$. Let $\chi' = \chi \cdot d_{Z(G)}$ for some $d \in D$. Let $\phi \otimes d$ be defined as $(\phi \otimes d)(g) = \phi(g)d(g)$, so that $\wh{\phi\otimes d}(\pi) = \wh\phi(\pi \otimes d)$. As such, we have
$$\mupl_{\chi}(\wh{\psi \otimes d}) = \mupl_{\chi'}(\wh \psi) \leq \epsilon,$$
since the Plancherel density function is unchanged by twisting by an unramified character.

Letting $\psi_{\chi}$ be the average of $\psi\otimes d$ with respect to $\chi$, and defining $\phi_{\chi}$ similarly, we see that $\phi_{\chi}$ and $\psi_{\chi}$ satisfy (1) and (2) in the statement of the proposition, completing the proof.
\end{proof}


\begin{thebibliography}{CMS90}
	
\bibitem[Art88]{Art88}
	J. Arthur,
	\emph{The invariant trace formula II. Global Theory},
	J. Amer. Math. Soc. \textbf{1} (1988),
	501-554.
	
\bibitem[Art89]{Art89}
	J. Arthur,
	\emph{The $L^2$-Lefschetz numbers of Hecke operators},
	Invent. Math. \textbf{97} (1989),
	 257-290.
	
\bibitem[Art02]{Art02}
	J. Arthur,
	\emph{A Stable trace formula. I. General expansions},
	J. Inst. Math. Jussieu \textbf{1} (2002),
	175-277.

\bibitem[Art13]{Art13}
	J. Arthur,
	\emph{The endoscopic classification of representations; orthogonal and symplectic groups},
	American Mathematical Society, Providence, RI, 2013.
	
\bibitem[AL70]{AL70}
	A.O.L. Atkin and J. Lehner,
	\emph{Hecke operators on $\Gamma_0(m)$},
	Math. Ann. \textbf{185} (1970),
	134-160.
	
\bibitem[AP05]{AP05}
	A.-M. Aubert and R. Plymen,
	\emph{Plancherel measure for GL(n, F) and GL(m, D): explicit formulas and Bernstein decomposition}.
	J. Number Theory \textbf{112} (2005),
	 26-66.
	
\bibitem[CD90]{CD90}
	L. Clozel and P. Delorme,
	\emph{Le th\'eor\`eme de Paley-Wiener invariant pour les groupes de Lie r\'educatifs. II},
	Ann. Sci. \'Ec. Norm. Sup\'er. \textbf{23} (1990),
	193-228.
	
\bibitem[CMS90]{CMS90}
	L. Corwin, A. Moy, and P. Sally,
	\emph{Degrees and formal degrees for division algebras and $\GL_n$ over a $p$-adic field},
	Pac. J. Math. \textbf{141} (1990), no. 1, 21-45.
	
\bibitem[FL13]{FL13}
	T. Finis and E. Lapid,
	\emph{An approximation principal for congruence subgroups}
	arXiv:1308.3604 [math.GR].
	
\bibitem[FL15]{FL15}
	T. Finis and E. Lapid,
	\emph{An approximation principle for congruence subgroups II: Application to the limit multiplicity problem}
	arXiv:1504.04795v1 [math.NT].
	
\bibitem[FLM14]{FLM14}
	T. Finis, E. Lapid, and W. Mueller,
	\emph{Limit multiplicities for principal congruence subgroups of $\GL(n)$ and $\SL(n)$},
	J. Inst. Math. Jussieu \textbf{14} (2014), 589-638.
	
	
\bibitem[Gel75]{Gel75}
	S. Gelbart,
	\emph{Automorphic forms on adele groups},
	Princeton University Press,
	Princeton (1975).
	
\bibitem[GJ79]{GJ79}
	S. Gelbart and H. Jacquet,
	Forms of $\GL(2)$ from the analytic point of view,
	In \emph{Automorphic forms, representations, and $L$-functions (Proc. Sympos. Pure Math., Oregon State Univ., Corvallis, Ore., 1977), Part 1},
	Proc. Sympos. Pure Math., XXXIII, pages 213-252.
	Amer. Math. Soc., Providence (1979).
	
	
\bibitem[KL06]{KL06}
	A. Knightly and C. Li,
	\emph{Traces of Hecke Operators},
	The American Mathematical Society,
	Providence (2006).

\bibitem[Pal12]{Pal12}
	M. Palm,
	\emph{Explicit $\GL(2)$ trace formulas and mixed Weyl laws}.
	arXiv:1212.4282v1 [math.NT].

\bibitem[RT11]{RT11}
	A. Raghuram and N. Tanabe,
	\emph{Notes on the arithmetic of Hilbert modular forms}.
	J. Ramanujan Math. Soc. \textbf{26} (2011), no. 3,
	pp. 261-319.
	
\bibitem[Sau97]{Sau97}
	F. Sauvageot,
	\emph{Principe de densit\'e pour les groupes r\'eductifs},
	Compos. Math. \textbf{108} (1997),
	151-184.
	
\bibitem[Sch02]{Sch02}
	R. Schmidt,
	\emph{Some remarks on local newforms for $\GL(2)$},
	J. Ramanujan Math. Soc. \textbf{17} (2002),
	115-147.

\bibitem[Ser80]{Ser80}
	J.-P. Serre,
	\emph{Trees},
	Springer-Verlag,
	Berlin-Heidelberg (1980).

\bibitem[Ser97]{Ser97}
	J.-P. Serre,
	\emph{R\'epartition asymptotic des values propres de l'op\'erateur de Hecke $T_p$},
	J. Amer. Math. Soc. \textbf{10} (1997), no. 1, 
	75-102.
	
\bibitem[Shi63]{Shi63}
	H. Shimizu,
	\emph{On Discontinuous Groups Operating on the Product of the Upper Half Planes},
	Ann. Math. \textbf{77} (1963), no. 1,
	33-71.
	
	
	
 \bibitem[Shi12]{Shi12}
	S.W. Shin,
	\emph{Automorphic Plancherel density theorem}.
	Israel J. Math. \textbf{192} (2012),
	 83-120.
	 
\bibitem[ST12]{ST12}
	S.W. Shin and N. Templier, 
	\emph{Sato-Tate theorem for families and low-lying zeros of automorphic $L$-functions},
	with appendices by Robert Kottwitz and Raf Cluckers, Julia Gordon, and Immanuel Halupczok,
	to appear in Invent. Math.
	
\bibitem[ST14]{ST13}
	S.W. Shin and N. Templier,
	\emph{On fields of rationality for automorphic representation},
	Compos. Math. \textbf{150} (2014), no. 12,
	2003-2053.

\bibitem[van72]{van72}
	G. van Dijk,
	\emph{Computation of certain induced characters of $p$-adic groups},
	Math. Ann. \textbf{199} (1972)
	229-240.
	
\bibitem[Wal03]{Wal03}
	J.-L. Waldspurger,
	\emph{La formule de Plancherel d'apr\`es Harish-Chandra}.
	J. Inst. Math. Jussieu \textbf{2}, (2003),
	235-333.
	
\bibitem[Wei09]{Wei09}
	J. Weinstein,
	\emph{Hilbert modular forms with prescribed ramification},
	International Mathematics Research Notices
	\textbf{2009} (2009), no. 8, 1388-1420.
	
\bibitem[Zel80]{Zel80}
	A.V. Zelevinsky,
	\emph{Induced representations of reductive $\pp$-adic groups. II. On irreducibly representations of $\GL(n)$},
	Ann. Sci. \'Ec. Norm. Sup\'er. \textbf{13} (1980), no. 2, 165-210.
	
\end{thebibliography}
\end{document}